\documentclass[12pt]{amsart}
\usepackage{amsmath,amsfonts,amsthm,amscd,stmaryrd, amssymb,mathrsfs}
\usepackage[all,cmtip]{xy}
\usepackage{enumitem}
\usepackage{hyperref}
\usepackage{tikz, tikz-cd}
\newtheorem{theorem}{Theorem}[section]

\newtheorem{proposition}[theorem]{Proposition}
\newtheorem{lemma}[theorem]{Lemma}
\newtheorem{corollary}[theorem]{Corollary}
\theoremstyle{definition}
\newtheorem{definition}[theorem]{Definition}
\newtheorem{remark}[theorem]{Remark}

\newcommand{\aaa}{\alpha}
\newcommand{\bbb}{\beta}
\newcommand{\ccc}{\gamma}
\newcommand{\CCC}{\Gamma}

\newcommand{\DDD}{\Delta}

\newcommand{\id}{{\rm{id}}}

\newcommand{\eee}{\epsilon}
\newcommand{\CP}{\mathbb{CP}}
\newcommand{\PP}{\mathbb{P}}
\newcommand{\CC}{\mathbb{C}}

\newcommand{\RR}{\mathbb{R}}
\newcommand{\ZZ}{\mathbb{Z}}
\newcommand{\QQ}{\mathbb{Q}}

\newcommand{\SL}{{\rm{SL}}}

\newcommand{\Sing}{{\rm{Sing}\,}}

\newcommand{\qdr}{\CP_1\times\CP_1}

\newcommand{\mf}{\mathfrak}

\newcommand{\ol}{\overline}

\newcommand{\lras}{\,\longrightarrow\,}

\newcommand{\set}{\,|\,}

\newcommand{\inv}{^{-1}}

\newcommand{\minus}{\backslash}
\newcommand{\ptl}{\partial}

\newcommand{\reg}{{\rm{reg}}}

\renewcommand{\tilde}{\widetilde}

\DeclareMathOperator{\Ker}{Ker}

\usepackage[letterpaper]{geometry}
\numberwithin{equation}{section}
\begin{document}
\title{Fibrations on the 6-sphere and Clemens threefolds}
\author{Nobuhiro Honda}\address{Department of Mathematics, Institute of Science Tokyo, Japan}\email{honda@math.titech.ac.jp}
\author{Jeff Viaclovsky}\address{Department of Mathematics, University of California, Irvine, USA}\email{jviaclov@uci.edu}
\begin{abstract} 
Let $Z$ be a compact, connected $3$-dimensional complex manifold with vanishing first and second Betti numbers and non-vanishing Euler characteristic. We prove that there is no surjective holomorphic mapping from $Z$ onto any $2$-dimensional complex space. In other words, $Z$ can only possibly fiber over a curve. 
This result applies in particular to a class of threefolds, known as Clemens threefolds, which are diffeomorphic to a connected sum $k \# (S^3 \times S^3)$ for $k \geq 2$.  This result also gives a new restriction on any hypothetical complex structure on the $6$-sphere $S^6$. 
\end{abstract}
\maketitle

\vspace{-7mm}
\section{Introduction}
The main theorem in this paper is the following which proves the non-existence of a surjective holomorphic mapping from a certain type of non-Moishezon threefold onto any $2$-dimensional complex space.
\begin{theorem} \label{thm:main1}
Let $Z$ be a compact, connected $3$-dimensional complex manifold satisfying  $b_1(Z) = b_2(Z)=0$ and $b_3(Z) \neq 2$, where $b_j(Z)$ denotes $j$th Betti number of $Z$.  Let $Y$ be any $2$-dimensional complex space. Then there does not exist any surjective holomorphic mapping $f : Z \rightarrow Y$. 
\end{theorem}
Of course, under the assumption that $b_1(Z) = b_2(Z) = 0$, $\chi_{\rm top}(Z) = 2 - b_3(Z)$, so the assumption $b_3(Z) \neq 2$ is equivalent to $\chi_{\rm top}(Z)\neq 0$, where $\chi_{\rm top}(Z)$ denotes the topological Euler characteristic of $Z$. 

Let $Z$ and $Y$ be irreducible complex spaces with $1 \leq \dim(Y) < \dim(Z)$.  A \textit{fibration} $f: Z \rightarrow Y$ is a proper surjective holomorphic mapping with connected fibers. We say that $Z$ \textit{fibers} over $Y$ and call $Y$ the \textit{base} of the fibration.
Note that singular fibers and fibers with dimension greater than the codimension $\dim(Z) - \dim(Y)$ are allowed. 
In terms of fibrations, Theorem \ref{thm:main1} has the following corollary. 
\begin{corollary}\label{cor:main1}
Let $Z$ be a compact, connected $3$-dimensional complex manifold satisfying  $b_1(Z) = b_2(Z)=0$ and $b_3(Z) \neq 2$. If $f : Z \rightarrow Y$ is a fibration, then $Y$ is a curve. In other words, $Y$ cannot be a surface. 
\end{corollary}
There are many interesting known examples of complex threefolds with $b_1(Z) = b_2(Z) = 0$; see for example \cite{CDP98, LRS18}. Some other examples are given by 
$Z_k := k\#(S^3\times S^3)$,
which is the $k$-fold connected sum of the product 6-manifold $S^3\times S^3$ for any $k\ge 1$. Then $Z_k$ satisfies $b_1(Z_k) =b_2 (Z_k) =0$ and $b_3(Z_k) = 2k$. Certain complex structures on $Z_k$ for $k \geq 2$ are known as Clemens complex structures; see \cite{Clemens, Friedman, FuLiYau, LuTian1996, Tiansmoothing}. Theorem~\ref{thm:main1} applies in particular to any complex manifold homeomorphic to $Z_k$, $k \geq 2$. Note that $Z_1=S^3 \times S^3$ satisfies $b_3(Z_1)= 2$, and admits a Calabi-Eckmann complex structure for which there exists a surjective holomorphic mapping $Z_1\rightarrow\qdr$; see \cite{CalabiEckmann}. 

\begin{remark}
\label{r:s6}
The existence of a complex structure on the six-dimensional sphere $S^6$ is an open problem, first pointed out by Kirchhoff and Hopf; see \cite{Kirchoff, Hopf}. We do not attempt to give a complete history of this problem here, but instead refer the reader to the survey \cite{A18} and the references therein. Some restrictions on a hypothetical complex structure on $S^6$ are known; see for example \cite{Angella, CDP98, CDP20, Gray97, LeBrunS6, LRS18, Ugarte}. While Theorem~\ref{thm:main1} does not resolve the Kirchhopf-Hopf Problem, it does give a new restriction on any hypothetical complex structure on~$S^6$, since $b_1(S^6) = b_2(S^6) = 0$ and $b_3(S^6) = 0 \neq 2$. 
\end{remark}

\begin{remark} It is worth pointing out that many examples of non-K\"ahler Calabi-Yau threefolds were constructed by Fino-Grantcharov-Vezzoni as torus fiber bundles over singular K3 surfaces; see \cite{FinoGrant}. This generalized earlier work of Fu-Yau where the base was a smooth K3 surface; see \cite{FuYau2007, FuYau2008}.  All of these examples have either $b_1(Z) \neq 0$ or $b_2(Z) \neq 0$, so of course there is no contradiction with Theorem~\ref{thm:main1}. But it is interesting that in the course of the proof of Theorem~\ref{thm:main1}, we indeed have to rule out the case that the base $Y$ is a singular K3 surface; see Section~\ref{s:a0case}. We also note that other interesting examples of non-K\"ahler Calabi-Yau threefolds were constructed in \cite{FinePanov}.
\end{remark}

\subsection{Outline of the paper}

In Section~\ref{s:Pre}, we prove a number of preliminary results. We begin with the observation that in order to prove Theorem \ref{thm:main1}, one can assume without loss of generality that $f : Z \rightarrow Y$ has connected fibers, and that the base $Y$ is reduced, irreducible and normal.  We emphasize that the base $Y$ can be singular. In fact, the greatest source of difficulty in our paper is the possible singularities of the base, especially since duality theorems are very subtle in the presence of singularities.  In Subsection~\ref{ss:ad}, we recall some basic results regarding the algebraic dimension. Using \cite[Theorem~2.1]{CDP20} and \cite[Theorem~3.1]{LRS18}, we conclude that $a(Y) \leq 1$; see Proposition~\ref{thm:main2}. 
In Subsection~\ref{ss:fd}, we recall a result about the dimension of the fibers of $f$, 
in Subsection~\ref{ss:dl}, we briefly discuss the discriminant locus, and in Subsection~\ref{ss:desing} we discuss some relevant results on desingularizations of complex spaces.

In Section~\ref{s:constraints},  we prove some analytic and topological constraints on $Y$.   
In Subsection~\ref{ss:topology}, we obtain some topological constraints. 
Proposition~\ref{prop:b1v} shows that the assumption $b_2(Z) = 0$ implies that the fibration must have generic fiber an elliptic curve, and the assumption $b_1(Z) = 0$ implies that any desingularization of the base $Y$ is K\"ahler. Proposition~\ref{p:b2bdd0} is one of our main tools; we use the 5 term exact sequence associated to the Leray spectral sequence for $f$ to obtain a crucial bound on the second Betti number of $Y$.  In Subsection~\ref{ss:sing}, we define and discuss the classes of singularities with which we must deal in the subsequent sections: \textit{rational tree}, \textit{unicyclic}, and \textit{elliptic tree singularities}. 
 
In Section~\ref{s:a0case}, we consider the case that $a(Y) = 0$. In this case, Proposition \ref{prop:K3} shows that $Y$ must be bimeromorphic to a K3 surface. Corollary \ref{c:cK3} shows that the only possible singularities of $Y$ are rational tree singularities. Our main estimate on the second Betti number of $Y$ is given in Proposition~\ref{p:b2bdd}. This is used in Proposition \ref{p:aneq0} to rule out this case. 

In Section~\ref{s:elliptic}, we prove a number of technical results about normal elliptic surfaces, which all rely on the foundational results of Kodaira on elliptic fibrations \cite{Kodaira1963}. In particular, we will show that in the case $a(Y) = 1$, then $Y$ can only have rational tree, unicyclic, or elliptic tree singularities.  In Proposition~\ref{p:b2}, we obtain a lower bound on the second Betti number of $Y$.  In Proposition~\ref{p:Clemens}, 
we prove a result about the existence of a strong deformation retraction of a neighborhood of a singular fiber of a fibration on a surface, using some well-known results in the case where the surface is non-singular. Next, we prove a technical result which will be needed in the subsequent section; see Lemma~\ref{l:b2rel}. 

In Section \ref{s:a1case}, we then consider the case that $a(Y) = 1$. This puts a strong restriction on $Y$; it must be elliptic over $\PP^1$; see Proposition~\ref{p:ell}. After a lot of work, we then obtain an important estimate for a relative Betti number involving the number of singularities of each type; see Corollary~\ref{cor:b2relY}. The estimates in Proposition~\ref{p:b2} and Corollary~\ref{cor:b2relY} are sufficiently sharp so that when combined with Proposition~\ref{p:b2bdd0}, our main result follows; see Proposition~\ref{p:aneq1}. 

\subsection{Acknowledgements} The first author was partially supported by JSPS KAKENHI Grant 22K03308. The second author was partially supported by NSF Grants DMS-2105478, DMS-2404195, and a Fellowship from the Simons Foundation. The authors would also like to thank the referees for their careful reading of the paper, and for their numerous suggestions which greatly improved the exposition of the paper.

\section{Preliminaries}
\label{s:Pre}
In this section, we begin with some general results about holomorphic mappings from a complex manifold to a complex space, which will be used throughout the paper. Recall that a {\em{complex space}} $(X, \mathcal{O}_X)$ is a Hausdorff locally ringed space with structure sheaf $\mathcal{O}_X$ which is an algebra over the constant sheaf $\CC$ and is locally isomorphic to a closed analytic subvariety of an open set in $\CC^n$. A morphism in this category of locally ringed spaces is called a {\textit{holomorphic mapping}}. For background on the theory of complex spaces, we refer the reader to \cite{GR}. For notation, typically, we will omit the structure sheaf when the context is clear. 

First, let $Z$ be a compact connected complex manifold, and $f : Z \rightarrow Y$ a surjective holomorphic mapping to a complex space $Y$. Note that by \cite[p.~171 (a)]{GR}, $Y$ must be irreducible, and then by \cite[p.~169]{GR}, $Y$ must be of pure dimension. A complex space is {\textit{reduced}} if all local rings have no nontrivial nilpotent elements; see \cite[p.~8]{GR}. 
 For a complex space $X$, we let $X_{\rm red}$ denote the reduction of $X$.  
 Any holomorphic mapping $f$ between possibly non-reduced complex spaces naturally induces a holomorphic mapping $f_{\rm red}$ between the reductions of the spaces. This has the same underlying mapping as the original mapping $f$. In our situation, since $Z=Z_{\rm red}$ as $Z$ is assumed to be a complex manifold, we obtain from $f:Z\rightarrow Y$ a holomorphic mapping $f_{\rm red}:Z\rightarrow Y_{\rm red}$.
This means that $f$ factors as $Z\stackrel{f_{\rm red}}\lras Y_{\rm red} \rightarrow Y$, where the latter mapping is the canonical mapping from the reduction of $Y$.
Therefore, for proving Theorem \ref{thm:main1}, we may suppose that the target space $Y$ is a reduced complex space, and in the rest of this paper we always assume it.

Then for any given $f:Z\rightarrow Y$ as above, we canonically obtain the following commutative diagram of holomorphic mappings between compact complex spaces
\begin{align}\label{cd:p0}
 \xymatrix{ 
Z \ar[r]^f \ar[d]_g \ar[dr]^{f'}&Y \\
W \ar[r]_h & Y'\ar[u]_{\nu}
}
\end{align}
where 
\begin{itemize}
\item
the mapping $\nu:Y'\rightarrow Y$ is the normalization of the reduced space $Y$,
\item
the mapping $f':Z\rightarrow Y'$ is the natural lift of $f$ to the normalization \cite[p.~164, Proposition]{GR}, which is a proper mapping because $Z$ is compact,
\item the sequence $Z\stackrel g\lras W\stackrel h\lras Y'$ is the Stein factorization of the proper mapping $f'$ \cite[p.~213]{GR}. 
\end{itemize}
In particular, the mapping $h$ is finite and 
$g$ has connected fibers.
Since $Y'$ is normal, so is $W$ \cite[p.~213]{GR}.
If the given mapping $f:Z\rightarrow Y$ is surjective,
then $g:Z\rightarrow W$ will be surjective.
Hence, to prove Theorem \ref{thm:main1}, we may assume that the mapping $f$ has connected fibers and the surface $Y$ is reduced, irreducible and normal. To simplify the assumptions, note that normality of $Y$ implies that $Y$ is reduced and locally irreducible. It follows that
for any normal space $Y$, irreducibility of $Y$ is equivalent to connectedness of $Y$. Therefore the assumptions that $Y$ is reduced, irreducible and normal are equivalent to the assumptions that $Y$ is connected and normal. 

\subsection{Algebraic dimension}
\label{ss:ad}
In this subsection, we summarize some properties of the algebraic dimension. For proofs, we refer the reader to \cite[Section~3]{U75}. 
\begin{definition} For an irreducible reduced compact complex space $X$, the  \textit{algebraic dimension} $a(X)$ of $X$ is the transcendence degree of the field $\CC(X)$ of meromorphic functions on $X$ over $\CC$.
\end{definition}
The algebraic dimension satisfies $0\le a(X)\le \dim X$ (see \cite[p.~217]{GR}) and is a bimeromorphism invariant. There is a smooth projective variety $V$ and a surjective meromorphic mapping $\Psi: X \dashrightarrow V$ which induces an isomorphism $\CC(V) \simeq \CC(X)$. The space $V$ is called an {\textit{algebraic reduction}} of $X$, and is unique up to bimeromorphic equivalence. We say that $X$ is \textit{Moishezon} if $a(X) = \dim(X)$. We note the following property, the proof of which is obvious.  
\begin{proposition}
\label{p:adim}
If $f : X_1 \rightarrow X_2$ is a surjective holomorphic mapping between irreducible reduced complex spaces $X_1$ and $X_2$, then $a(X_2) \leq a(X_1)$. 
\end{proposition}
The following result is stated in the proof of \cite[Corollary~2.3]{CDP20}; we provide a proof for the reader's convenience. 
\begin{proposition}
\label{aZ2}
Let $Z$ be a 3-dimensional connected compact complex manifold with algebraic dimension $a(Z)=2$. Then there exists a non-holomorphic meromorphic mapping $g:Z\dashrightarrow \PP^1$.
\end{proposition}
\begin{proof}
First, let $S\subset\PP^N$ be any non-degenerate smooth surface
and $p\in S$ any point.
Let $L\subset\PP^N$ be a codimension two linear subspace that satisfies $L\cap T_pS=\{p\}$.
Then the intersection $S\cap L$ at $p$ is transverse.
Let $\pi:\PP^N\dashrightarrow \PP^1$ be the projection from $L$. Then the restriction $\pi|_S$ is a meromorphic mapping from $S$ onto $\PP^1$.
This is not defined at $p$ and cannot be extended holomorphically across $p$.
Thus, there exists a meromorphic mapping $\pi:S\dashrightarrow\PP^1$ that is not defined at $p$.

Next, let $\Phi:Z\dashrightarrow S$ be an algebraic reduction of $Z$ so that $S$ is a non-degenerate smooth projective surface. 
Let $B\subset Z$ be the indeterminacy locus of the meromorphic mapping $\Phi$.
This may be an empty set and by a theorem of Remmert (see \cite[Theorem~2.5]{U75}), its codimension in $Z$ is at least two.
When $B=\emptyset$, choose a point $p\in S$ outside the discriminant locus of $\Phi$ and also a meromorphic mapping $\pi:S\dashrightarrow\PP^1$ that is not defined at $p$.
Then obviously, the composition $\pi\circ\Phi$ is not defined at any point of the fiber $\Phi\inv(p)$ and it cannot be extended holomorphically across this fiber.
When $B\neq\emptyset$, by Sard's theorem applied to the holomorphic and hence smooth mapping $Z\minus B\stackrel{\Phi}\lras S$, there exists a point $z\in Z\minus B$ such that the differential $(df)_z$ is surjective. Then again taking a meromorphic mapping $\pi: S\dashrightarrow\PP^1$ that is not defined at $p:=\Phi(z)$,
the composition $\pi\circ\Phi$ is not defined at $z$.
\end{proof}

We next recall the following important result, which is proved in \cite[Theorem~2.1]{CDP20} and \cite[Theorem~3.1]{LRS18}.
\begin{theorem}[\cite{CDP20, LRS18}]
\label{thm:CDP}
Let $Z$ be a 3-dimensional connected compact complex manifold with $b_2(Z) = 0$. Assume that there exists a non-holomorphic meromorphic mapping $g: Z \dashrightarrow \PP^1$. Then 
$\chi_{\rm top}(Z)=0$, where $\chi_{\rm top}(Z)$ denotes the topological Euler characteristic of $Z$.
\end{theorem}

Theorem~\ref{thm:CDP} immediately excludes the possibility that the space $Y$ satisfies $a(Y)=2$:
\begin{proposition}\label{thm:main2}
Let $Z$ be a compact, connected $3$-dimensional complex manifold satisfying  $b_2(Z)=0$ and
$\chi_{\rm top}(Z)\neq 0$. Let $Y$ be a normal compact complex surface. If there exists a surjective holomorphic mapping $f:Z\to Y$, then $a(Y)<2$.
\end{proposition}
\begin{proof}
Since $f$ is surjective, by Proposition \ref{p:adim}, $a(Z)\ge a(Y)$. So if $a(Y)=2$, then $a(Z)\ge 2$.
Since any Moishezon manifold $X$ satisfies $b_{2}(X)>0$ (see for example \cite[Lemma 1.4]{CDP98})), this means $a(Z)=2$ under the present assumption $b_2(Z)=0$. By Proposition~\ref{aZ2}, there exists a non-holomorphic meromorphic mapping $g:Z\dashrightarrow \PP^1$. By Theorem \ref{thm:CDP}, this means $\chi_{\rm top}(Z)=0$, which contradicts our assumption. Hence $a(Y)\neq 2$. 
\end{proof}
We note that Proposition~\ref{thm:main2} is the only place in the paper that we use the assumption that $\chi_{\rm top}(Z) \neq 0$. 
\subsection{Fiber dimension}
\label{ss:fd}
 The \textit{fiber} of $f$ over a point $y$ is denoted by $Z_y$ and is the set $f^{-1}(y)$ with structure sheaf $\mathcal{O}_Z / f^* \mathfrak{m}_y$, 
where $\mathfrak{m}_y$ is the maximal ideal in the local ring $\mathcal{O}_{Y,y}$ at $y$.  
We have the following basic result on the dimension of the fibers of $f$. As in \cite{U75}, we define  $\dim(f) = \dim(Z) - \dim(Y)$.
\begin{proposition}
\label{p:fd}
 If $f: Z \rightarrow Y$ is a fibration with $Z$ irreducible, then $\dim_z(Z_y) \geq \dim(f)$ for all $y \in Y$ and all $z \in Z_y$. Furthermore, the set of $y \in Y$ for which there exists $z \in f^{-1}(y)$ with $\dim_z(Z_y) > \dim(f)$ is a subvariety of $Y$ of codimension at least $2$.
\end{proposition}
\begin{proof}
This is proved in \cite{Remmert1957}; see also \cite[Chapter~3]{Fischer}.
\end{proof}
Applied to our situation with $\dim(Z) = 3$ and $\dim(Y) =2$, we conclude that all fibers have pure dimension~$1$, except for a finite subset of $Y$ over which the fibers can have an irreducible component of dimension~$2$. 
\subsection{Discriminant locus}
\label{ss:dl}
For a reduced complex space $Y$, we denote $\Sing Y$ and $Y_{\reg}$ for the singular locus and smooth locus of $Y$, respectively. There exists a thin analytic subset $\mathscr{D} \subset Y$ with $\Sing Y \subset \mathscr{D}$,  called the {\textit{discriminant locus}} of $f$, such that the restriction 
\begin{align}\label{fb1}
Z\setminus f\inv (\mathscr{D}) \stackrel{f'}\lras Y\setminus  \mathscr{D}
\end{align}
is a smooth submersion; see  \cite[Chapter~L]{GunningII}, \cite{Peternell_Remmert}, or \cite[Proposition~2.2]{HV2025}.
By Ehresmann's fibration theorem the mapping \eqref{fb1} is a fiber bundle mapping of smooth manifolds; see \cite{Eh51}.  
Note that since $\mathscr{D}$ is an analytic subset of $Y$ and $Y$ is irreducible, 
the complement $Y\minus \mathscr{D}$ is connected by \cite[p.~145]{GR}. Hence, all fibers of the mapping \eqref{fb1} are mutually diffeomorphic. We will refer to a fiber of the mapping in \eqref{fb1} as a \textit{general fiber}. 

\subsection{Desingularization}
\label{ss:desing}
Given a reduced complex space $Y$, a \textit{desingularization} of $Y$ is 
a smooth complex manifold $\tilde{Y}$ together with a holomorphic bimeromorphic mapping
$\mu: \tilde Y \rightarrow Y$.  We say that a desingularization $\mu:\tilde Y\rightarrow Y$ of $Y$ is {\textit{strict}}
if $\mu\inv(Y_{\reg})\stackrel{\mu}\lras Y_{\reg}$ is an isomorphism. Such a desingularization exists for any reduced complex space $Y$ by Hironaka \cite{Hiro71}.
The next proposition enables us to make a base change for a holomorphic mapping from a complex manifold to a reduced complex space under a strict desingularization of the target space, and which will be used in Section~\ref{s:constraints} to obtain a strong constraint on $Y$.
\begin{proposition}\label{p:bim}
Let $f:Z\rightarrow Y$ be a surjective holomorphic mapping from a connected compact complex manifold $Z$ to a reduced compact complex space $Y$.
Then for any strict desingularization $\mu:\tilde Y\rightarrow Y$ of $Y$, there exist a connected complex manifold $\tilde Z$ and the following commutative diagram
\begin{align}\label{cd:bch}
  \begin{CD}
     \tilde Z @>{\tilde f}>> \tilde Y \\
  @V{\phi}VV    @VV{\mu}V \\
     Z   @>f>>  Y
  \end{CD}
\end{align}
where $\tilde f$ is a surjective holomorphic mapping and $\phi$ is a bimeromorphic holomorphic mapping such that its restriction to $(f\circ \phi)\inv (Y_{\reg})$ is an isomorphism onto $f\inv(Y_{\reg})$.  
\end{proposition}
\begin{proof}
The fiber product (in the category of sets)
\begin{align}
Z\times_Y \tilde Y=
\big\{(z,\tilde y)\in Z\times \tilde Y  \set f(z)= \mu(\tilde y)\}
\end{align}
is an analytic subset of the product complex manifold $Z\times \tilde Y$, which is not necessarily irreducible.
Let $E\subset \tilde Y$ be the exceptional locus of $\mu$ and
$\pi_Z$ and $\pi_{\tilde Y}$ the natural projections from $Z\times_Y\tilde Y$ to $Z$ and $\tilde Y$ respectively. Then since $\mu$ gives an isomorphism $\tilde Y\minus E\simeq Y_{\reg}$ as $\mu$ is strict, $\pi_Z$ gives a biholomorphic mapping from $\pi_{\tilde Y}\inv (\tilde Y\minus E)$ to $f\inv(Y_{\reg})$.
Let $Z'$ be the irreducible component of $Z\times_Y\tilde Y$ which contains the former open subset $\pi_{\tilde Y}\inv (\tilde Y\minus E)$.
Then $\pi_Z$ gives a bimeromorphic holomorphic mapping from $Z'$ to $Z$ but $Z'$ is not necessarily non-singular.
So let $\tilde Z\rightarrow Z'$ be a strict desingularization of $Z'$, which again exists by~\cite{Hiro71}.
Since $Z'$ does not have a singularity on the open subset $\pi_{\tilde Y}\inv (\tilde Y\minus E)\simeq f\inv(Y_{\reg})\subset Z$ as $Z$ is assumed to be non-singular, the desingularization does not make any change on $\pi_{\tilde Y}\inv (\tilde Y\minus E)\subset Z'$ from the strictness of the desingularization.

Define the holomorphic mappings $\tilde f$ and $\phi$ to be the compositions $\tilde Z\rightarrow Z'\stackrel{\pi_{\tilde Y}}\lras \tilde Y$ and $\tilde Z\rightarrow Z' \stackrel{\pi_Z}\lras Z$ respectively.
Then the commutativity of the diagram \eqref{cd:bch} is easily seen from the construction and $\phi$ is bimeromorphic since it is a composition of bimeromorphic mappings.
Moreover, $\phi$ gives an isomorphism $(f\circ \phi)\inv (Y_{\reg})\simeq f\inv(Y_{\reg})$ because the above isomorphism $\pi_{\tilde Y}\inv (\tilde Y\minus E)\simeq f\inv(Y_{\reg})$ induced by $\pi_Z$ is transferred to the required isomorphism through the strict desingularization $\tilde Z\rightarrow Z'$.
\end{proof}

\section{Constraints on $Y$}
\label{s:constraints}
In this Section, we will prove some analytical and topological constraints on $Y$. We will also discuss the constraints on the possible singularities of $Y$. 

\subsection{Topological constraints}
\label{ss:topology}
In this section, we prove several topological constraints on~$Y$. First, we recall some details regarding the Leray spectral sequence. 
Let $M$ and $N$ be any topological spaces, and  $f : M \rightarrow N$ a continuous mapping. The Leray spectral for the constant sheaf $\RR$ on $M$ satisfies 
\begin{align}
E_2^{p,q} = H^p(N,R^qf_*\RR) \Longrightarrow H^{p+q}(M,\RR);
\end{align}
see \cite[Theorem~F.5]{GunningIII}. 
The associated $5$ term exact sequence is
\begin{align}
0 \lras E_2^{1,0} \lras H^1(M,\RR) \lras E_2^{0,1} 
\overset{d_2}{\lras} E_2^{2,0} \lras H^2(M, \RR);
\end{align}
see \cite[Chapter~5]{Weibel}.
In case the fibers of $f$ are connected and $f$ is surjective, then $f_* \RR \simeq \RR$, and this sequence becomes 
\begin{align}
\label{Leray5}
0 \lras H^1(N, \RR) \overset{f^*}{\lras} H^1(M,\RR) \lras H^0(N, R^1 f_* \RR) 
\overset{d_2}{\lras} H^2(N,\RR) \overset{f^*}\lras H^2(M, \RR).
\end{align}
We will use the following proposition which gives constraints for the topology of general fibers of $f:Z\rightarrow Y$ and $Y$. 
\begin{proposition}\label{prop:b1v}
Let $f : Z \rightarrow Y$ be a surjective holomorphic mapping with $Z$ a compact complex threefold, and $Y$ a normal compact complex surface. If $f$ has connected fibers, then the following hold. 
\begin{itemize}
\item [(i)] If $b_2(Z)=0$, then the general fiber of $f$ is an elliptic curve.
\item [(ii)] If $b_1(Z)=0$, then $b_1(Y)=b_1(\tilde Y)=0$, where $\tilde Y$ is any desingularization of $Y$.
In particular, $\tilde Y$ admits a K\"ahler metric.
\end{itemize}
\end{proposition}
\begin{proof} Since $f$ is surjective, by Proposition~\ref{p:fd}, the generic fiber of $f$ is a curve. 
Let $C$ be a fiber of \eqref{fb1}, which obviously has trivial normal bundle.  
Since $b_2(Z)=0$, any smooth complex line bundle over $Z$ is a torsion bundle. In particular, so is the canonical bundle $K_Z$. Therefore, from the adjunction formula, $K_C\simeq K_Z|_C$ is also a torsion bundle. Because $C$ is a compact connected Riemann surface, $H^2(C,\ZZ)\simeq\ZZ$. Hence, $K_C$ is topologically trivial, which means that $C$ is an elliptic curve.

For Item (ii), since $H^1(Z,\RR)=0$ as $b_1(Z)=0$, we obtain $b_1(Y)=0$ from the exact sequence \eqref{Leray5}. To prove that $b_1(\tilde Y) = 0$, we may assume that the desingularization $ \mu : \tilde Y\rightarrow Y$ satisfies $\mu\inv(Y_{\reg})\stackrel{\mu}\lras Y_{\reg}$ is an isomorphism. 
Since $f$ is surjective, we can then use Proposition~\ref{p:bim} to obtain the diagram \eqref{cd:bch}.
Since $\phi$ is bimeromorphic and both $Z$ and $\tilde{Z}$ are complex manifolds, by \cite[Proposition 2]{Kawai1965}, the homomorphism $H_1(\tilde Z,\ZZ) \simeq H_1(Z,\ZZ)$ induced by $\phi$ is an isomorphism.
Hence $b_1(\tilde Z) = b_1(Z)$, and this means $b_1(\tilde Z) =0$ as $b_1(Z)=0$ by assumption. Since the original mapping $f$ has connected fibers, so does the mapping $\tilde{f}$; this follows from \cite[Corollary~1.12]{U75}.
Hence, again $\tilde f_*\RR\simeq\RR$ for the direct image sheaf.
Therefore, by \eqref{Leray5} applied to $\tilde f$, we obtain an injection
$H^1(\tilde Y,\RR) \hookrightarrow H^1(\tilde Z,\RR).$
Since $b_1(\tilde Z) =0$ as above, this means $b_1(\tilde Y) = 0$.  Since $\tilde{Y}$ is a smooth compact surface with $b_1(\tilde{Y}) = 0$, from \cite[Theorem~IV.3.1]{BHPV}, it admits a K\"ahler metric. 
\end{proof}
The next result gives an estimate on the second Betti number of $Y$, and is crucial in proving Theorem \ref{thm:main1} in Sections \ref{s:a0case} and \ref{s:a1case}.
\begin{proposition}\label{p:b2bdd0}
Let $f:Z\rightarrow Y$ be a surjective holomorphic mapping with connected fibers from a 3-dimensional compact complex manifold $Z$ to a normal surface $Y$ and suppose $b_1(Z) = b_2(Z) = 0$.
Then 
\begin{align}\label{b2Y}
b_2(Y) = \dim H^0\big(Y,R^1f_*\RR\big).
\end{align}
Let $\Sigma$ be any closed subset of $Y$ and $f' := f|_{Z \setminus f^{-1}(\Sigma)}$. Then we have an inequality
\begin{align}\label{b2Y0}
b_2(Y) \le b_2(Y,Y\minus\Sigma) - b_1(Y\minus\Sigma) + \dim H^0\big(Y\minus\Sigma, R^1 f'_*\RR\big),
\end{align}
where $b_2(Y,Y\minus\Sigma):=\dim H_2(Y,Y\minus\Sigma\,;\RR)$.
\end{proposition}
\begin{proof}
Since $f : Z \rightarrow Y$ has connected fibers, the sequence \eqref{Leray5} is
\begin{align}\label{}
0\lras H^1(Y)\lras H^1(Z) \lras H^0\big(Y,R^1f_*\RR\big) \stackrel{d_2}\lras H^2(Y) \lras H ^2(Z),
\end{align}
where we have omitted the coefficient ring $\RR$.
Since $H^1(Z)=H^2(Z)=0$ by the assumptions $b_1(Z) = b_2(Z)=0$, we obtain from this an isomorphism
\begin{align}\label{d1iso}
d_2: H^0(Y,R^1f_*\RR) \stackrel{\simeq}\lras H^2(Y)
\end{align}
for the differential of the spectral sequence.
Hence we obtain the equality \eqref{b2Y}.

Next, we use the naturality of the Leray spectral sequence with respect to open inclusions; see \cite[Theorem~1.3]{HV2025}. This yields the commutative diagram
\begin{align}
\label{e:CD}
  \begin{CD}
     H^0\big(Y,R^1f_*\RR \big) @>{d_2}>> H^2(Y,\RR) \\
  @V{\aaa}VV    @VV{\bbb}V \\
     H^0\big(Y\minus\Sigma,R^1f'_*\RR \big)   @>{d'_2}>>  H^2(Y\minus\Sigma,\RR) 
  \end{CD}
\end{align}
where $d'_2$ is a differential for the $E_2$-term in the spectral sequence associated to $f'$, and $\aaa$ and $\bbb$ are the restriction homomorphisms, respectively.  
From this, we obtain $d_2(\Ker \aaa) \subset \Ker \bbb$.
Since $d_2$ is an isomorphism from \eqref{d1iso}, this means
\begin{align}\label{ie:b2bdd0}
\dim (\Ker\aaa) \le \dim(\Ker\bbb).
\end{align}

Next, associated with the pair $(Y,\Sigma)$, again omitting the coefficient ring $\RR$ for all cohomology groups, we obtain a long exact sequence 
\begin{align}\label{rellong1}
H^1(Y) \lras H^1(Y\minus\Sigma) \lras H^2(Y,Y\minus\Sigma)
\lras
H^2(Y) \stackrel{\bbb}\lras H^2(Y\minus\Sigma). 
\end{align}
From Proposition \ref{prop:b1v} (ii), $H^1(Y) =0$ as $b_1(Z)=0$.
From \eqref{rellong1}, this means that 
\begin{align}\label{ab}
\dim (\Ker \bbb)
= b_2(Y,Y\minus\Sigma) - b_1(Y\minus\Sigma).
\end{align}
From \eqref{ie:b2bdd0}, this means 
\begin{align}\label{}
\dim (\Ker\aaa) &\le b_2(Y,Y\minus\Sigma) - b_1(Y\minus\Sigma).
\end{align}
This implies an inequality
\begin{align*}
\dim H^0\big(Y,R^1f_*\RR \big) &\le \dim(\Ker\aaa) + \dim H^0\big(Y\minus\Sigma,R^1f'_*\RR \big)\\
&\le b_2(Y,Y\minus\Sigma) - b_1(Y\minus\Sigma)
+ \dim H^0\big(Y\minus\Sigma,R^1f'_*\RR \big).
\end{align*}
From \eqref{b2Y},
this implies the desired inequality \eqref{b2Y0}.
\end{proof}

\subsection{Constraints on the singularities of $Y$}
\label{ss:sing}
We next introduce the following kinds of surface singularities which will arise in the proof of Theorem~\ref{thm:main1}.  
We first recall the following notion concerning resolutions of normal surface singularities; see for example \cite[p.~134]{Dur79} or \cite[p.~50]{Dimca}. 
\begin{definition}
\label{d:gsing}
Let $Y$ be a normal complex surface and $y\in Y$ a singularity of $Y$. A resolution $\mu:\tilde Y\rightarrow Y$ of $y$ is called {\em very good} if the exceptional divisor $E=\mu\inv(y)$
satisfies the following four properties:
(i)   any irreducible component of $E$ is non-singular, (ii)  any intersection of two distinct components of $E$ is transverse, (iii) any two distinct components are either disjoint or intersect in exactly one point, and (iv) no three components meet at a single point.
\end{definition}
For background on resolutions of normal surface singularities, we refer the reader to \cite{GH, Ishii, LaBook}. Any normal surface singularity is isolated.
There exists a unique resolution $\mu_1:\tilde Y_1\rightarrow Y$ which is minimal in the sense that any resolution $\mu:\tilde Y\rightarrow Y$ factors through $\mu_1$; namely, there exists a unique holomorphic mapping $\aaa:\tilde Y\rightarrow \tilde Y_1$ such that $\mu = \mu_1\circ\aaa$. This minimality is equivalent to the condition that no $(-1)$-curve is contained in the exceptional curve of $\mu_1$.
The minimal resolution is not necessarily a very good resolution, but we can always obtain a very good resolution from the minimal resolution after finitely many blowups;
see for example \cite[p.~496]{Brieskorn1986}, \cite[Chapter~8B]{Mumford1976}, or \cite[II.7]{BHPV}.

There exists a unique minimal very good resolution which is minimal amongst the collection of all very good resolutions. Note that, as  \cite[Example~4.1]{HV2025} illustrates, the minimal very good resolution \textit{can} contain $(-1)$-curves in the exceptional divisor. To any very good resolution, we can associate the {\textit{dual graph}} of the resolution, denoted by $\Gamma$, which has a vertex for each irreducible component of $E$, with an edge between two vertices if they intersect. We next introduce the following kinds of surface singularities which appear in Theorem~\ref{thm:main1}. 

\begin{definition}\label{d:tsing}
Let $Y$ be a normal complex surface and $y\in Y$ a singularity of $Y$. Let $\mu:\tilde Y \rightarrow Y$ be a very good resolution of $y$ and $E :=\mu\inv(y)$ the exceptional divisor of $\mu$. We say that 
\begin{itemize}
\item $y$ is a {\em rational tree singularity} of $Y$ if all irreducible components of $E$ are rational, and the dual graph of $E$ is a tree. 
\item $y$ is a {\em unicyclic singularity} of $Y$ if all irreducible components of $E$ are rational, and the dual graph of $E$ contains exactly one cycle. 
\item $y$ is an {\em elliptic tree singularity} of $Y$ if $E$ has exactly one irreducible component that is a smooth elliptic curve, all other components are rational,
and the dual graph of $E$ is a tree. 
\end{itemize}
\end{definition}
Since there exists a minimal very good resolution, it follows that if one of these conditions holds for \textit{some} very good resolution, then the same holds for \textit{all} very good resolutions.

\begin{remark}Artin defined a \textit{rational singularity} in \cite{Artin1966}. These are necessarily rational tree singularities (see \cite[Lemma~1.3]{Brieskorn1967}), but it is easy to see that the class of rational tree singularities is strictly larger than the class of rational singularities; see \cite[Example~4.2]{HV2025}. {\textit{Cusp singularities}} are those with a resolution with exceptional divisor a cycle of rational curves (or a nodal rational curve), so the class of unicyclic singularities contains the class of cusp singularities \cite{Hirzebruch1971}.  A {\textit{simple elliptic singularity}} is one for which the exceptional curve is a single elliptic curve with negative self-intersection \cite{Saito1974}. Wagreich defined an  {\textit{elliptic singularity}} in \cite{Wagreich1970}.  It is possible for any of the types in Definition~\ref{d:tsing} to be elliptic in this sense: \cite[Example~4.2]{HV2025} is a rational tree singularity which is elliptic, any cusp singularity (which is a unicyclic singularity) is elliptic, and any simple elliptic singularity (which is an elliptic tree singularity) is elliptic.
\end{remark}

The next proposition gives an alternative characterization of the singularities in Definition~\ref{d:tsing}, in terms of the first Betti number of the exceptional curve of a very good resolution.
\begin{proposition}
\label{p:scri}
Under the assumptions in Definition \ref{d:tsing}, we have the following:
\begin{itemize}
\item 
$y$ is a rational tree singularity if and only if $b_1(E) = 0$. 
\item
$y$ is a unicyclic singularity if and only if $b_1(E)=1$.
\item 
$y$ is an elliptic tree singularity if and only if  $b_1(E)=2$ and some irreducible component of $E$ is a smooth elliptic curve.
\end{itemize}
\end{proposition}
\begin{proof}
Let $\CCC$ be the dual graph of $E$. If $b$ is the number of closed cycles in $\CCC$ and $g$ is the sum of the genera of the irreducible components of $E$, then by \cite[Proposition~2.3.1]{Dimca}, 
\begin{align}\label{b1E}
b_1(E) = 2g + b.
\end{align}
The proposition easily follows from this formula; the details are omitted. 
\end{proof}

The link $L$ of the singularity is a compact $3$-manifold such that a neighborhood of the singularity is homeomorphic to the cone over $L$; see \cite[Chapter~1]{Dimca}. 
We will also need the following characterization of rational tree singularities.
\begin{proposition}
\label{p:scri2}
Let $Y$ be a normal complex surface and $y\in Y$ a singularity of $Y$.
Then $y$ is a rational tree singularity if and only if $b_1(L) = 0$, that is, $L$ is a $\QQ$-homology $3$-sphere. 
\end{proposition}
\begin{proof}
Let $\mu : \tilde{Y} \rightarrow Y$ be a very good resolution with exceptional divisor $E = \mu^{-1}(y)$. Then by \cite[Proposition~2.3.4]{Dimca}, $b_1(L) = b_1(E) = 0$, so this follows immediately from Proposition~\ref{p:scri}. 
\end{proof}
The following proposition will be used often. 
\begin{proposition}
\label{p:Lamotke}
Let $Y$ be a normal surface and $y\in Y$ a singularity, which is isolated,
$\mu:\tilde Y\lras Y$ be any resolution of $y$, and $E$ be the exceptional curve of $\mu$.
Then there exists an open neighborhood $U$ of $y$ in $Y$ such that $\tilde{U}: = \mu^{-1}(U)$ admits a strong deformation retraction onto $E$.
\end{proposition}
\begin{proof}
If the resolution is very good, this is proved in 
\cite[p.~174, Proposition 6]{Lamotke}; see also \cite{Mumford1961}.
For an arbitrary resolution $\mu:\tilde Y\rightarrow Y$ of $y$, 
take a domination 
$\tilde  Y_1 \stackrel{\aaa}\lras \tilde Y \stackrel{\mu}\lras Y$,
 where the composition $\mu\circ\aaa$ is a very good resolution of $y$ and $\aaa$ is a composition of ordinary blowups. 
Let $E_1$ be the exceptional curve of the resolution $\mu_1:=\mu\circ\aaa:\tilde Y_1\rightarrow Y$.
As above, there exists a neighborhood $U$ of $y$ in $Y$ such that $\tilde U_1 := \mu_1\inv(U)$ admits
a strong deformation retraction onto $E_1$. Since $\alpha$ is a proper mapping, it is straightforward to show that $\tilde{\Phi}$ induces the required strong deformation retraction; see for example \cite[Proposition~3.1]{HV2025}.
\end{proof}
Recall that the algebraic dimension of $Y$ could be $0$ or $1$, since the case that $a(Y)=2$ has already been ruled out. 
The following result puts constraints on the singularities of $Y$ and will be proved in Section~\ref{s:a0case} in the case that $a(Y) =0$ and in Section~\ref{s:a1case} in the case that $a(Y) =1$. 
\begin{theorem} 
\label{thm:ay}
Let $f:Z\rightarrow Y$ be a surjective holomorphic mapping with connected fibers from a 3-dimensional connected compact complex manifold $Z$ satisfying $b_1(Z) = b_2(Z) = 0$ to a normal compact surface $Y$.
If $a(Y) = 0$, then $Y$ can have only rational tree singularities.   If $a(Y) = 1$, then $Y$ can only have rational tree, unicyclic, or elliptic tree singularities. 
\end{theorem}

\section{The case $a(Y) = 0$}\label{s:a0case}
In this section, we consider the case of $a(Y) = 0$ in Theorem~\ref{thm:main1}.
\begin{proposition}\label{prop:K3}
Let $f:Z\rightarrow Y$ be a surjective holomorphic mapping with connected fibers from a 3-dimensional connected compact complex manifold $Z$ satisfying $b_1(Z) = b_2(Z) = 0$ to a normal compact surface $Y$.
If $a(Y) = 0$, then $Y$ is bimeromorphic to a K3 surface.
\end{proposition}
\begin{proof}
Since $a(\tilde Y) = 0$ for any desingularization $\tilde Y$ of $Y$ by assumption,
the classification of compact complex surfaces \cite[p.~257]{BHPV} means that $\tilde Y$ is bimeromorphic to a complex torus, a class VII surface, or a K3 surface. 
But since $b_1(\tilde Y) = 0$ from Proposition \ref{prop:b1v}, 
the first two possibilities cannot happen and hence $\tilde Y$ can only be bimeromorphic to a K3 surface.
\end{proof}
Thus, we need to investigate the structure of curves on a normal K3 surface whose algebraic dimension is zero. We first consider the case of a smooth K3 surface.
\begin{proposition}
\label{prop:cK3}
Let $Y_0$ be a non-singular K3 surface with $a(Y_0) = 0$. If $C$ is a connected curve on $Y_0$, namely if $C$ is a 1-dimensional subvariety of $Y_0$ which is connected, then $C$ is the exceptional divisor of the minimal resolution of a rational double point of a surface.
\end{proposition}
\begin{proof}
Suppose that $D$ is a non-zero effective divisor on $Y_0$ that is not necessarily reduced.
Let $L = [D]$ denote the line bundle associated to $D$. 
Since $Y_0$ is a K3 surface, the Riemann-Roch formula gives 
\begin{align}
\label{rr}
\chi(Y_0,L) =  2 + \frac{1}{2} D^2.
\end{align}
Since $a(Y_0) = 0$, any line bundle on $Y_0$ has at most one section, up to scaling. Therefore, since $D$ is a non-zero effective divisor, we must have $h^0(Y_0,L) = 1$, and $h^0(Y_0,L^*) = 0$. Hence, using Serre duality, $\chi(Y_0,L)=1-h^1(Y_0,L)$. From \eqref{rr}, $D^2 = -2 ( 1 + h^1(Y_0,L))$. In particular,
\begin{align}
\label{ineq}
D^2 \leq -2. 
\end{align}
for any non-zero effective divisor $D$ that is not necessarily reduced.  If $D$ is an arbitrary non-zero divisor on $Y_0$, we can write $D = D_1 - D_2$,
where both $D_1$ and $D_2$ are effective, at least one of $D_1$ and $D_2$ is not zero, and they have no common component. Then since $D_1D_2\ge 0$,
\begin{align}
D^2 = (D_1+D_2)^2 -4D_1D_2\le (D_1+D_2)^2.
\end{align}
From \eqref{ineq}, $(D_1+D_2)^2\le -2$. 
Therefore, $D^2\le -2$ and we obtain that any non-zero effective divisor on $Y_0$ has a negative self-intersection number.
If the support of $D$ is connected, by Grauert's criterion, this means that the divisor $D$ can be contracted to a point complex analytically, and since $K_{Y_0}$ is trivial, the conclusion follows from \cite[Proposition~III.2.5]{BHPV}. 
\end{proof}
Note that Proposition \ref{prop:cK3} does not imply that the singularities of every normal K3 surface $Y$ with $a(Y)=0$ are rational double points.
Instead, we have: 
\begin{corollary}\label{c:cK3}
Suppose that $Y$ is a normal complex surface that is bimeromorphic to a smooth K3 surface with algebraic dimension zero. Then $Y$ has at most only rational tree singularities (recall Definition \ref{d:tsing})
and there is a diagram 
\begin{align}\label{cd:p}
\xymatrix{ 
& \tilde Y \ar[ld]_{\mu} \ar[rd]^{\bbb} & \\
Y & & Y_0
}
\end{align}
where $\mu$ is a very good resolution,
$\bbb$ is a composition of blowdowns of $(-1)$-curves, and $Y_0$ is a smooth K3 surface with $a(Y_0)=0$.
\end{corollary}
\begin{proof}
Let $\mu:\tilde Y\rightarrow Y$ be a very good resolution of all singularities of $Y$, 
and let
$\bbb:\tilde Y\rightarrow Y_0$ be a composition of the blowdowns of $(-1)$-curves such that $Y_0$ has no $(-1)$-curve.
By assumption, $Y_0$ is a (smooth) K3 surface with $a(Y_0)=0$. Thus we obtain the diagram \eqref{cd:p}.

Proposition~\ref{prop:cK3} in particular means that any connected component of the union of all curves on $Y_0$ is a tree of smooth rational curves.
Since $\bbb$ is just a composition of ordinary blowups, this means that any connected component of the union of all curves on $\tilde Y$ is also a tree of smooth rational curves.
Therefore, so are the connected components of the exceptional curves of $\mu$.
Then every singularity of $Y$ is a rational tree singularity, by definition. 
\end{proof}
The next proposition will be used soon.
\begin{proposition} 
\label{p:blk3}
Under the situation of the previous corollary,
let $E$ and $F$ be the exceptional divisors of $\mu$ and $\bbb$ respectively. Then 
\begin{align*}
b_2(\tilde Y) = b_2(Y) + b_2(E) = b_2(Y_0) + b_2(F).
\end{align*}
Moreover, any connected curve on $\tilde Y$ can be complex analytically contracted to a smooth point or a rational tree singularity.
\end{proposition}
\begin{proof}
The equality $b_2(\tilde Y)  = b_2(Y_0) + b_2(F)$ is obvious since $\bbb$ is just the composition of ordinary blowups.
Assume that $\Sing Y \neq \emptyset$, to show $b_2(\tilde Y) = b_2(Y) + b_2(E)$, 
let $p_1,\dots, p_k$ be all the singularities of $Y$ and
$E_1,\dots, E_k$ the exceptional curves over $p_1,\dots, p_k$ respectively.
From Corollary \ref{c:cK3},
each $E_i$ is a tree of smooth rational curves.
For simplicity, we put $E=E_1+\dots+E_k$.
We denote $\ol Y:=Y/E$, which is the surface obtained from $\ol Y$ by identifying all points of $E$.
Then the quotient map $\tilde Y\rightarrow \ol Y$ factors as 
$$
\tilde Y\stackrel{\mu}\lras Y \lras \ol Y,
$$
where the latter mapping is a quotient map that identifies $p_1,\dots, p_k$. If $k>1$, then $\ol Y$ is non-normal and the mapping $Y\rightarrow\ol Y$ is the normalization.
In the sequel, we omit the coefficient ring $\RR$ for all homology groups.
Recall that a \textit{good pair} $(X,A)$ is a topological space $X$ with a nonempty closed subspace $A$ that is a deformation retract of some neighborhood in $X$; see \cite[p.~114]{Hatcher}. By the conical structure theorem from Milnor \cite[Theorem~2.10]{Milnor_sing} (see also \cite{BV} for the analytic case), 
$(Y,\{p_1,\dots, p_k\})$ is a good pair. Hence, by \cite[Theorem~2.13]{Hatcher}, we have an exact sequence
\begin{align}
\cdots \lras \tilde H_q(\{p_1,\dots, p_k\})\lras \tilde H_q(Y) \lras \tilde H_q (\ol Y) \lras \tilde H_{q-1}(\{p_1,\dots, p_k\}) \lras \cdots
\end{align}  
for the reduced homology groups. From this, we obtain the isomorphism
\begin{align}\label{0reli1}
H_q(\ol Y)\simeq H_q(Y),\quad q>1.
\end{align}
Similarly, using Proposition \ref{p:Lamotke}, the pair $(\tilde{Y}, E)$ is a good pair, 
so we obtain an exact sequence
\begin{align}\label{0les1}
\cdots
\lras H_2(E) \lras 
H_2(\tilde Y)\lras H_2(\ol Y)
\lras 
H_1(E) \lras \cdots.
\end{align}
Since the intersection matrix for each $E_i$ is negative definite and $E_i\cap E_j=\emptyset$ if $i\neq j$, the homology classes of any irreducible components of $E$ are linearly independent in $H_2(\tilde Y)$.
This means that the homomorphism $H_2(E) \rightarrow 
H_2(\tilde Y)$ in the exact sequence \eqref{0les1} is injective.
On the other hand, $H_1(E)=\oplus_{i=1}^k H_1(E_i) = 0$ because all $E_1,\dots, E_k$ are trees of smooth rational curves.
Hence, from the exact sequence \eqref{0les1}, 
$b_2(\tilde Y) = b_2(\ol Y) + b_2(E).$ From \eqref{0reli1}, this means 
\begin{align}\label{0betti1}
b_2(\tilde Y) = b_2(Y) + b_2(E).
\end{align}

Finally, from Proposition~\ref{prop:cK3}, any irreducible curve on $Y_0$ is a component of the exceptional divisor of the minimal resolution of a rational double point singularity, so is in particular a smooth rational curve. For the same reason, the dual graph of 
any connected curve on $Y_0$ is a tree. Let $\tilde C$ be any connected curve on $\tilde Y$.
Since $\tilde Y$ is obtained from $Y_0$ by ordinary blowups, it follows that all irreducible components of $\tilde C$ are smooth rational curves, and the dual graph of $\tilde{C}$ is also a tree. Furthermore, if $\tilde D$ is the maximal connected curve on $\tilde Y$ that contains $\tilde C$, then $\tilde D$ is contracted to a rational double point via $\beta:\tilde Y\to Y_0$.
Hence, by Grauert's criterion, 
the intersection matrix of $\tilde{D}$ is negative definite.
Therefore, the intersection matrix of the subcurve $\tilde{C}$ is also negative definite.
Hence, again by Grauert's criterion, the connected curve $\tilde{C}$ can be complex analytically contracted to a point which is a smooth point or a rational tree singularity. 
\end{proof}
Returning to the situation of Theorem~\ref{thm:main1} with the reduction procedure in Section~\ref{s:Pre} being applied, let $f:Z\rightarrow Y$ be a surjective holomorphic mapping with connected fibers from a 3-dimensional compact complex manifold $Z$ with $b_1(Z)=b_2(Z)=0$ onto a normal compact complex surface $Y$ with $a(Y)=0$. 
By Proposition~\ref{prop:b1v} (i), a regular fiber of $f$ is an elliptic curve.
By Proposition~\ref{prop:K3}, $Y$ is bimeromorphic to a K3 surface $Y_0$ with $a(Y_0)=0$.
Since any compact smooth complex surface with algebraic dimension zero has at most a finite number of curves by \cite{Kr75}, 
so does our surface $Y$ because if not, then a desingularization of $Y$ would have an infinite number of curves.
\begin{definition}
\label{d:Sigma} The subset $\Sigma \subset Y$ is the discriminant locus of $f$ as in  \eqref{fb1} together with {\em all curves} in $Y$. In other words, $\Sigma$ is the union of all curves on $Y$ and all isolated discriminant points of $f$. We let $D  := f\inv (\Sigma) \subset Z$. 
\end{definition}
Note that $Y\minus\Sigma$ is connected as $Y$ is connected and $\Sigma$ is a proper analytic subset of $Y$. The set $D$ is a proper analytic subset of $Z$. Since all fibers of $f$ are at least $1$-dimensional by Proposition~\ref{p:fd}, if $D$ is non-empty, its irreducible components are either 1-dimensional or 2-dimensional. We denote $f' = f_{Z \minus D}$, and then $f':Z\minus D\rightarrow Y\minus \Sigma$ is still a $T^2$-bundle differentiably.

The next proposition plays a key role in proving Theorem \ref{thm:main1} in the case $a(Y)=0$.
\begin{proposition}\label{p:b2bdd}
Let $f:Z\rightarrow Y$ be a surjective holomorphic mapping with connected fibers from a 3-dimensional compact complex manifold $Z$ with $b_1(Z)=b_2(Z)=0$ onto a normal compact complex surface $Y$. Let $n_Y$ be the number of irreducible components of all curves on $Y$.
Then we have an inequality
$$
b_2(Y) \le n_Y + 2.
$$
\end{proposition}
\begin{proof}
As above, the mapping $f':Z\minus D\rightarrow Y\minus\Sigma$ is differentiably a $T^2$-bundle. Therefore, over the open subset $Y\minus\Sigma$ of $Y$, the sheaf $R^1f'_*\RR$ is a locally constant sheaf and its fiber at a point $y\in Y\minus\Sigma$ can be identified with the cohomology group $H^1(f\inv(y),\RR)$, which is isomorphic to $H^1(T^2, \RR)\simeq\RR^2$.
This sheaf is determined by a representation $\rho : 
\pi_1(Y \setminus \Sigma) \rightarrow PSL(2, \ZZ)$ (see \cite[Section~IV.9]{Iversen}) 
and the space of sections is identified with the invariant elements under the monodromy 
\cite[Proposition~5.5.14]{DavisKirk}.
Consequently, $\dim H^0(Y\minus\Sigma,R^1 f'_*\RR) \le 2$. 
From \eqref{b2Y0} in Proposition \ref{p:b2bdd0}, this means
\begin{align}\label{b2Y1}
b_2(Y) \le b_2(Y,Y\minus\Sigma) + 2.
\end{align}
So to prove the proposition, it suffices to show $b_2(Y,Y\minus\Sigma) \le n_Y$.

In the following, the coefficient ring of all cohomology groups are $\RR$. If $\Sigma = \emptyset$, then there are no curves on $Y$, and the inequality is $b_2(Y) \leq 2$, which follows immediately from \eqref{b2Y1}. 
If $\Sigma \neq \emptyset$, let $\Sigma_1,\dots, \Sigma_k$ be the connected components of $\Sigma$ which are (complex) 1-dimensional (if there are any such components) and $p_1,\dots, p_l$ the zero-dimensional components of $\Sigma$ (if there are any isolated points in $\Sigma$).  
From Corollary~\ref{c:cK3}, 
the point $p_i$ ($1\le i\le l$) is either a smooth point of $Y$
or a rational tree singularity of $Y$.
Take sufficiently small open neighborhoods $N_1,\dots, N_k$ of $\Sigma_1,\dots, \Sigma_k$ respectively and also sufficiently small open neighborhoods $M_1,\dots, M_l$ of $p_1,\dots, p_l$ respectively so that all these are disjoint and
\begin{align}\label{ds01}
H^2 (Y,Y\minus\Sigma) \simeq \Big(\bigoplus_{i=1}^k H^2 (N_i,N_i\minus\Sigma_i)\Big)
\oplus \Big(\bigoplus_{i=1}^l H^2 (M_i,M_i\minus\{p_i\})\Big)
\end{align}
holds by excision.

First, we show $H^2 (M_i,M_i\minus\{p_i\})=0$ for any $i$ for some choice of the neighborhood $M_i$. This is obvious if $p_i$ is a smooth point of $M_i$.
Hence, we suppose that $M_i$ has a singularity at $p_i$, which is a rational tree singularity as above.
We consider the long exact sequence associated with the pair $(M_i,M_i\minus\{p_i\})$, which is
\begin{align}\label{rc01}
\cdots\lras
H^1 (M_i)\lras
H^1 (M_i\minus\{p_i\})
\lras
H^2 (M_i,M_i\minus\{p_i\})\lras
H^2 (M_i)\lras\cdots.
\end{align}
Let $L_i$ be the link of the singularity $p_i$.
By the conical structure theorem, the neighborhood $M_i$ can be taken in such a way that it is homeomorphic to the cone over the link $L_i$.
This means $H^q(M_i\minus\{p_i\})\simeq H^q(L_i)$ for any $q\ge 0$.
Further, since $p_i$ is a rational tree singularity, by Proposition \ref{p:scri2},
$L_i$ is a $\QQ$-homology 3-sphere.
Furthermore, since $M_i$ is contractible to the one point $p_i$ by the conical structure theorem, $H^q(M_i) = 0$ for any $q>0$.
Therefore, from \eqref{rc01}, we obtain 
$
H^2 (M_i,M_i\minus\{p_i\}) \simeq H^1(L_i) = 0.
$

Next, we calculate the first direct summand of \eqref{ds01} in a similar way.
By the above choice, $N_i$ can have singularities, but all of them belong to the 1-dimensional connected component $\Sigma_i$ of $\Sigma$.
Let $\tilde N_i\rightarrow N_i$ be a very good resolution of all singularities of $N_i$ (if any), and $\tilde\Sigma_i$ the inverse image (not the proper transform) of $\Sigma_i$ by the resolution mapping. Then 
since $\tilde\Sigma_i$ is obviously connected, by Proposition \ref{p:blk3}, it can be contracted to a smooth point or a rational tree singularity.
Let $\tilde N_i\rightarrow \widehat N_i$ be such a contraction of $\tilde\Sigma_i$ and $q_i$ the image point of $\tilde\Sigma_i$ by the contraction mapping. Then we have natural isomorphisms
\begin{align}\label{isoms}
N_i\minus\Sigma_i\simeq \tilde N_i\minus\tilde\Sigma_i
\simeq \widehat N_i\minus\{q_i\}.
\end{align}
Further, by the conical structure theorem, if $L_i$ is the link of $q_i$, then $\widehat N_i\minus\{q_i\}$ can be deformation retracted onto $L_i$, including the case where $q_i$ is a smooth point of $\widehat N_i$.
So from~\eqref{isoms}, $N_i\minus\Sigma_i$ is homotopic to $L_i$, and by Proposition~\ref{p:scri2}, this is a $\QQ$-homology 3-sphere since $q_i$ either a smooth point or a rational tree singularity.

We then consider the long exact sequence of cohomology groups associated with the pair $(N_i,N_i\minus\Sigma_i)$, which is
\begin{align}\label{rc02}
\cdots\lras
H^1 (N_i\minus\Sigma_i)\lras
H^2 (N_i,N_i\minus\Sigma_i)\lras
H^2 (N_i)\lras
H^2 (N_i\minus\Sigma_i)\lras\cdots.
\end{align}
Since $N_i\minus\Sigma_i$ is homotopic to a $\QQ$-homology 3-sphere as above, $H^1 (N_i\minus\Sigma_i)=H^2 (N_i\minus\Sigma_i)=0$.
Therefore, from \eqref{rc02}, for any $i$, we obtain an isomorphism
$
H^2 (N_i,N_i\minus\Sigma_i)\simeq H^2 (N_i).
$
Further, since $N_i$ admits a strong deformation retraction onto $\Sigma_i$ by Proposition~\ref{p:Lamotke}, $H^2 (N_i)\simeq H^2(\Sigma_i)$. Therefore we obtain for any $i$
\begin{align}
H^2 (N_i,N_i\minus \Sigma_i )\simeq H^2(\Sigma_i).
\end{align}
Hence, from \eqref{ds01}, we obtain 
\begin{align}
\label{ds02}
H^2 (Y,Y\minus\Sigma) \simeq \bigoplus_{i=1}^l H^2(\Sigma_i).
\end{align}
Since $\Sigma_i$ is a curve, $\dim H^2(\Sigma_i)$ is equal to the number of the irreducible components of $\Sigma_i$.
Since $\Sigma$ includes all curves on $Y$ as in Definition \ref{d:Sigma},
this implies $\dim H^2 (Y,Y\minus\Sigma) = n_Y$.
From \eqref{b2Y1}, this means $b_2(Y)\le n_Y + 2$.
\end{proof}
From the reduction procedure given in Section \ref{s:Pre}, the following proposition completes a proof of Theorem \ref{thm:main1} when $a(Y)=0$.
\begin{proposition}\label{p:aneq0}
Let $f:Z\rightarrow Y$ be a holomorphic mapping with connected fibers from a 3-dimensional connected compact complex manifold $Z$ satisfying $b_1(Z) = b_2(Z) = 0$ to a normal compact surface $Y$ with $a(Y)=0$. Then the mapping $f$ cannot be surjective.
\end{proposition}
\begin{proof}
Suppose that the mapping $f:Z\rightarrow Y$ is surjective.
By Proposition \ref{prop:K3}, this implies that $Y$ is bimeromorphic to a K3 surface and by Proposition \ref{p:b2bdd}, $Y$ satisfies 
\begin{align}\label{b2Y2}
b_2(Y)\le n_Y+2,
\end{align}
where $n_Y$ is as before the number of irreducible components of all curves on $Y$.
Consider again the diagram \eqref{cd:p}, which was
\begin{align}\label{cd:p2}
\xymatrix{ 
& \tilde Y \ar[ld]_{\mu} \ar[rd]^{\bbb} & \\
Y & & Y_0
}
\end{align}
By Proposition \ref{p:blk3}, if $E$ and $F$ are as before the exceptional divisors of $\mu$ and $\bbb$ respectively, then $b_2(\tilde Y) = b_2(Y) + b_2(E)$.
Hence, from \eqref{b2Y2}, $b_2(\tilde Y) \le n_Y + 2 + b_2(E)$.
On the other hand, from the same proposition, $b_2(\tilde Y) = b_2(Y_0) + b_2(F)$.
From these, using $b_2(Y_0)=22$ as $Y_0$ is a smooth K3 surface, we obtain $n_Y+b_2(E)- b_2(F)\ge 20$.
But $n_Y+b_2(E)- b_2(F)$ is exactly the number of irreducible components of all curves on the smooth K3 surface $Y_0$.
Since the intersection matrix formed by all irreducible components of curves on $Y_0$ is negative definite by Proposition \ref{prop:cK3}, this contradicts that the signature of the intersection form on $H^{1,1}(Y_0,\RR)$ on any K3 surface is $(1,19)$.
Therefore, the mapping $f$ cannot be surjective.
\end{proof}

\begin{remark} In the case of $S^6$, if one is willing to use \cite[Theorem~2.2]{CDP20}, 
then $a(S^6) = 0$, and one could stop at this point. We emphasize that the arguments given in the subsequent sections do also work in the case of $S^6$ to rule out the case that $a(Y) =1$. So while  \cite[Theorem~2.1]{CDP20} or \cite[Theorem~3.1]{LRS18} is crucial (to obtain the estimate $a(Y) \leq 1$), \cite[Theorem~2.2]{CDP20} is not needed for the proof of Theorem~\ref{thm:main1}. 
\end{remark}

\begin{remark} We also point out that the known complex structures on the Clemens threefolds $Z_k$ for $k \geq 2$ 
constructed in \cite{Clemens, Friedman, LuTian1996, Tiansmoothing} all satisfy $h^{0,1}(Z_k) = 0$.
Since $H^1(Z_k; \ZZ) = H^2(Z_k; \ZZ) = 0$, this implies that ${\rm{Pic}}(Z_k)$ is trivial. Consequently, these complex structures do not admit any effective divisor, so necessarily $a(Z_k) = 0$. So for these particular complex structures, the proof of Theorem~\ref{thm:main1} is complete.  However, we emphasize that Theorem~\ref{thm:main1} holds for any complex manifold $Z_k$ which is \textit{homeomorphic} to $k \# S^3 \times S^3$; there is no reason that an arbitrary complex structure on $Z_k$ must satisfy $h^{0,1}(Z_k) = 0$. It is an interesting problem to find complex structures on $Z_k$ for $k \geq 2$ satisfying $h^{0,1}(Z_k) \neq 0$. 
\end{remark}

\section{Normal elliptic surfaces}  
\label{s:elliptic}
In the section, we prove a number of results on normal elliptic surfaces (allowing singularities) which will be used in Section~\ref{s:a1case}. 
We begin with the following proposition.
\begin{proposition}\label{p:a1s}
Let $Y$ be a normal irreducible compact complex surface with $a(Y)=1$, then the following hold:
\begin{enumerate}
\item[(i)]
There exists a surjective holomorphic mapping $\pi:Y\rightarrow C$ onto a smooth compact curve $C$  whose generic fiber is a smooth elliptic curve.
 
\item[(ii)]
Every connected divisor on $Y$ is contained in a fiber of $\pi$.

\item[(iii)]
Any singularity of $Y$ is either a rational tree singularity, a unicyclic singularity, or an elliptic tree singularity (recall Definition~\ref{d:tsing}).

\item[(iv)] If $b_1(Y) = 0$, then $C \simeq \PP^1$. 
\end{enumerate}
\end{proposition}
\begin{proof}
Let $\mu:\tilde Y\rightarrow Y$ be a birational holomorphic mapping that gives a very good resolution for all singularities of $Y$.
Then since $a(\tilde Y)=a(Y)=1$, there exists a smooth projective curve $C$ and a surjective holomorphic mapping $\tilde\pi:\tilde Y\rightarrow C$ with connected fibers whose generic fibers are (smooth) elliptic curves \cite[Proposition~VI.5.1]{BHPV}.
As shown in the proof of the same proposition,
any connected divisor on $\tilde Y$ is contained in a fiber of $\tilde \pi$. In particular, every irreducible component of the exceptional divisor of $\mu$ is contained in a fiber.
This means that, around each singularity of $Y$, the composition $\tilde\pi\circ\mu\inv$ defined on $Y\minus\Sing Y$ can be regarded as a holomorphic function in a punctured neighborhood of the singularity.
By Riemann's extension theorem on a normal complex space \cite[p.~144]{GR}, this has a unique holomorphic extension to the singularity. 
This means that $\tilde\pi\circ\mu\inv$ is a holomorphic mapping from $Y$ to $C$.
Let $\pi$ be this mapping. 
Obviously, a generic fiber of $\pi$ is also a smooth elliptic curve.
Moreover, there exists no divisor on $Y$ that is mapped onto $C$ by $\pi$ because this is true for $\tilde\pi$.

In the following we will freely use Kodaira's classification of fibers on a smooth elliptic surface; see \cite[Theorem~6.2]{Kodaira1963}. To prove (iii), we consider the following commutative diagram
\begin{align}\label{cd:p4}
\xymatrix{ 
Y\ar[rd]_{\pi} & \tilde Y\ar[l]_{\mu} \ar[r]^{\bbb} \ar[d]^{\tilde\pi} & Y_0\ar[ld]^{\pi_0}\\
& C &
}
\end{align}
where $\pi_0:Y_0\rightarrow C$ is a relatively minimal model of $\tilde \pi$ and $\bbb$ is a composition of the blowdowns of $(-1)$-curves in fibers of $\tilde\pi$.
Let $E$ be any connected component of the exceptional divisor of $\mu$ and put $p:=\tilde\pi(E)\in C$ and $\tilde F:= \tilde\pi\inv(p)$. 
All irreducible components of $E$ are smooth because $\mu$ is a very good resolution.
Since the geometric genus of any irreducible component of a fiber of $\pi_0$ is at most one by Kodaira's classification of the singular fiber, the same is true for $\tilde\pi$. 
In particular, all irreducible components of $\tilde F$ and therefore also of $E$ have geometric genus at most one.

Suppose that the fiber $\tilde F$ has a component which is a smooth elliptic curve. Then the fiber $\pi_0\inv(p)$ has to be a smooth elliptic curve (including the case of multiple one) and 
the elliptic component in $\tilde F$ has to be the strict transform of $\pi_0\inv(p)$.
Since $\bbb$ is a composition of ordinary blowups, this means that $\tilde F$ is a tree of smooth curves and all components except for the elliptic component of $\tilde F$ are rational curves.
This means that, if $E$ does not contain the elliptic component, then the point $\mu(E)$ is a rational tree singularity, and otherwise the point $\mu(E)$ is an elliptic tree singularity.

Next, suppose that the fiber $\tilde F$ does not have an elliptic component.
This means that the fiber $\pi_0\inv(p)$ is a singular fiber, where a multiple smooth fiber is understood to be a smooth fiber.
Again by Kodaira's classification, all irreducible components of $\pi_0\inv(p)$
have a geometric genus zero and therefore so do 
the irreducible components of the fiber $\tilde F$.
Further, all irreducible components of $\tilde F$ are smooth rational curves except for possibly one component which could be a rational curve having exactly one node or exactly one cusp that comes from a type I$_1$ fiber or a type II fiber of $\pi_0$, respectively.
If we remove the singular component from $\tilde F$, then each connected component will be a tree of smooth rational curves.
This implies that for both cases $\mu(E)$ is a rational tree singularity.
Alternatively, if all irreducible components of $\tilde F$ are smooth, then they are rational curves, but $\tilde{F}$ can have a cycle that comes from type I$_b$ fiber ($b>1$) of $\pi_0$, a triple point coming from type IV fiber of $\pi_0$, or two components tangential at a point coming from type III fiber of $\pi_0$.
If $\tilde F$ has a triple point, then again since $\mu$ is very good, $E$ cannot contain all three components meeting at the point, and $E$ has to be a tree of smooth rational curves. Hence again $\mu(E)$ is a rational tree singularity.
A similar argument yields the same conclusion for the case of a type III fiber. 
Suppose finally that $\tilde F$ contains a cycle. If we remove any one irreducible component of this cycle from $\tilde F$, then the resulting curve is a tree of smooth rational curves. This means that if $E$ contains the entire cycle, then $\mu(E)$ is a unicyclic singularity, and otherwise $\mu(E)$ is a rational tree singularity.

For the last item (iv), since the fibers of $\pi$ are connected, the Leray spectral sequence for $\pi$ gives
\begin{align}
0 \lras H^1(C;\RR) \lras H^1(Y;\RR), 
\end{align}
so $b_1(Y) = 0$ implies that $b_1(C) =0$, and since $C$ is smooth, we have $C \simeq \PP^1$. 
\end{proof}
Motivated by this proposition, we make the following definition. 
\begin{definition} 
\label{d:nenu}
Let $Y$ be a normal compact complex surface with $a(Y) =1$. 
Then $n_e$ and $n_u$ will denote the number of elliptic tree singularities and unicyclic singularities on the surface $Y$, respectively.
\end{definition}
We will also need the following estimate for the second Betti number of a compact normal elliptic surface.
\begin{proposition}\label{p:b2}
Let $Y$ be a compact normal surface with $a(Y)=1$ and $\pi:Y\rightarrow\PP^1$ an elliptic fibration. 
Suppose that $b_1(\tilde Y) = 0$ for the minimal resolution $\tilde Y$ of the singularities of $Y$. Then 
\begin{align}\label{2ndbetti1}
b_2(Y)\ge 2p_g \big(\tilde Y\big) + 2 n_e + n_u + 2.
\end{align}
If moreover $\pi$ has at least one fiber with two or more irreducible components, then 
\begin{align}\label{2ndbetti2}
b_2(Y)\ge 2p_g \big(\tilde Y\big)  + 2 n_e + n_u +  3.
\end{align}
\end{proposition}
\begin{proof}
In the same way to the proof of Proposition \ref{p:blk3}, using the notations there, we have isomorphism $H_q(\ol Y)\simeq H_q(Y)$ for any $q>1$.
Also, we have the long exact sequence
\begin{multline}\label{les1}
\cdots\lras H_2(E) \lras 
H_2(\tilde Y)\lras H_2(\ol Y)
\lras 
H_1(E) \lras 
H_1(\tilde Y)\lras \cdots,
\end{multline}
where the homomorphism $H_2(E) \rightarrow H_2(\tilde Y)$ is injective, and $H_1(\tilde Y) =0$ from the assumption.  Therefore, 
\begin{align}\label{betti1}
b_2(Y)=  b_2(\tilde Y) - b_2(E) + b_1(E).
\end{align}

For the first Betti number of each connected component $E_i$ of $E$,  
from Proposition \ref{p:scri},
\begin{itemize}
\item $b_1(E_i) = 2$ if and only if $\mu(E_i)$ is an elliptic tree singularity,
\item $b_1(E_i) = 1$ if and only if $\mu(E_i)$ is a unicyclic singularity,
\item $b_1(E_i) = 0$ otherwise; namely if and only if $\mu(E_i)$ is a rational tree singularity.
\end{itemize}
From this, $\dim H_1(E) = 2n_e + n_u.$ 
Hence, from \eqref{betti1}, 
\begin{align}\label{betti2}
b_2(Y)=  b_2(\tilde Y) - b_2(E) + 2n_e + n_u.
\end{align}

We show that, in general, for a smooth K\"ahler elliptic surface $\pi:X\rightarrow C$ over a compact smooth curve $C$, if $\mf f_1,\dots, \mf f_k$ are all reducible fibers of $\pi$, then letting $\mf f'_1,\dots,\mf f'_k$ be curves obtained from $\mf f_1,\dots,\mf f_k$ by removing any one of the irreducible components respectively so that one irreducible component of $\mf f_i$ is not included in $\mf f'_i$, then
the cohomology classes represented by 
\begin{itemize}
\item
a fiber of $\pi$ and any one of K\"ahler forms on $X$
\item 
all the irreducible components of $\mf f'_1,\dots, \mf f'_k$
\end{itemize}
are linearly independent in $H^2(X)$.
Writing $\ccc_{ij}$ for the cohomology classes of the irreducible components of $\mf f'_i$ ($1\le i\le k$), $\omega$ for a K\"ahler class on $X$, and $F$ for the fiber class of $\pi$, we assume the relation
$$
a\omega + b F + \sum_{i=1}^k \sum_j c_{ij} \ccc_{ij} = 0
$$
for some $a,b,c_{ij}\in\RR$.
Taking the intersection number of this relation with $F$, we obtain $a=0$. From this, taking the self-intersection number of this relation, we obtain 
\begin{align}\label{Z}
\sum_{i=1}^k \Big(\sum_j c_{ij} \ccc_{ij}\Big)^2 = 0.
\end{align}
Now we recall Zariski's lemma; see \cite[Lemma~III.8.2]{BHPV}. This says that if $\mf f = \sum_{j=1}^r m_j C_j$ is any reducible fiber of a (proper) fibration on a smooth surface 
over a smooth curve, then a divisor $D$ of the form $\sum_{j=1}^r  n_j C_j$ always satisfies $D^2\le 0$ and the equality holds only when $D = q \mf f$ for some $q\in \QQ$.
Since we are removing an irreducible component from $\mf f_i$ for any $i$, 
from \eqref{Z}, this means $c_{ij}=0$ for any $j$ and $i$.
Thus, we obtain the required linearly independency.
Furthermore, all these classes belong to the subspace $H^{1,1}(X,\RR)$.

We apply this to the present elliptic surface $\tilde \pi:\tilde Y\rightarrow \PP^1$.
Let $\mf f_1,\dots, \mf f_k$ be all reducible fibers of $\tilde\pi$ and $r_1,\dots, r_k>1$ be the number of their irreducible components respectively.
Then from the Hodge decomposition for $H^2(\tilde Y,\RR)$ and
the above linearly independency, we obtain an inequality
\begin{align}\label{betti5}
b_2(\tilde Y) \ge \sum_{i=1}^k(r_i - 1) + 2 + 2p_g(\tilde Y),
\end{align}
where $p_g(\tilde Y) = \dim H^0(\tilde Y,K_{\tilde Y})$.
Since any component of the exceptional curve $E$ of $\mu:\tilde Y\rightarrow Y$ is in some reducible fiber from Proposition \ref{p:a1s} and since any $\mf f_i$ has at least one component that is not in $E$, we have an inequality
\begin{align}\label{betti6}
b_2(E) \le \sum_{i=1}^k(r_i - 1),
\end{align}
and the equality holds if and only if all fibers of the original elliptic fibration $\pi:Y\rightarrow \PP^1$ are irreducible.
From \eqref{betti5}, this means an inequality
\begin{align}\label{betti3}
b_2(\tilde Y) \ge b_2(E) + 2 + 2p_g(\tilde Y),
\end{align}
and the equality holds only if all fibers of $\pi$ are irreducible.
From \eqref{betti2}, this gives an inequality
\begin{align}\label{betti4}
b_2(Y) \ge 2 + 2p_g(\tilde Y) + 2n_e + n_u,
\end{align}
and the equality holds only if all fibers of $\pi:Y\rightarrow C$ are irreducible.
This gives the desired estimates \eqref{2ndbetti1} and \eqref{2ndbetti2}.
\end{proof}

We next prepare some topological results on a neighborhood of a singular fiber of a proper holomorphic fibration over a disc.  Let $\DDD\subset\CC$ be a unit disk, and for any positive real number $\eee<1$, let $\DDD_{\eee}^*:=\{z\in\DDD\set 0<|z|<\eee\}$
and $\ol\DDD_{\eee} \setminus \{0\}:=\{z\in\CC\set 0<|z|\le \eee\}$.
The next proposition is well-known in the case that $N$ is smooth.
\begin{proposition}\label{p:Clemens}
Let $N$ be a normal complex surface and $\pi:N\rightarrow \DDD$ a proper surjective holomorphic mapping. Suppose that $N$ is non-singular away from $F_0:=\pi\inv(0)$ and that $f$ is of maximal rank on $N\minus F_0$. Then:
\begin{itemize}
\item[(i)] there exists a strong deformation retraction $\Phi:[0,1]\times N\rightarrow N$ of $N$ onto the fiber $F_0$ that is a lift of the standard deformation retraction from $\DDD$ onto $\{0\}$,
\item[(ii)] 
if $F_0 = \sum_j m_j F_{0,j}$ is the decomposition of the fiber $F_0$ into irreducible components, then 
the isomorphism $H_2(N)\simeq H_2(F_0)$ induced from the deformation retraction in (i) maps the fiber class of $\pi$ to 
the class $\sum_j m_j[F_{0,j}]$, where $[F_{0,j}]$ denotes the fundamental class of the component $F_{0,j}$,
\item[(iii)]
for any $0<\eee<1$,
there exists a strong deformation retraction of $\pi\inv(\ol\DDD_{\eee} \setminus \{0\})$ onto $\pi\inv (\ptl \DDD_{\eee})$.
\end{itemize}
\end{proposition}
\begin{proof}
For the first assertion, by first taking the minimal resolution of all singularities of $Y$ (which belong to $F_0$ by assumption) and next applying blowups to eliminate non-transversal intersections in $F_0$, we obtain a smooth surface $\tilde N$ equipped with a birational holomorphic mapping $\mu:\tilde N\rightarrow N$, such that the composition $\tilde\pi:=\pi\circ\mu:\tilde N\rightarrow \DDD$ satisfies the following conditions:
\begin{itemize}
\item any irreducible component of
the fiber $\tilde F_0:=\tilde \pi\inv(0)$ is non-singular,
\item any intersection of two distinct components of $\tilde F_0$ is transverse,
\item $\tilde F_0$ has no triple point.
\end{itemize}
Of course, the fiber $\tilde F_0$ could have a non-reduced component in general. 
Then by using \cite[(5.1) and Theorem 6.9]{Cl77}
(see also \cite{Lamotke1975}, \cite[III.14]{BHPV} and \cite{Lefschetz}), 
there exists a strong deformation retraction $\tilde\Phi:[0,1]\times\tilde N\rightarrow \tilde N$ of $\tilde N$ onto the fiber $\tilde F_0$, which is a lift of the standard retraction from $\DDD$ onto a point $\{0\}$ defined by $(t,z)\mapsto tz$, where $t\in[0,1]$ and $z\in\DDD$; in particular, $\tilde\Phi$ preserves the fibration $\tilde\pi$ in the sense that for any $t\in[0,1]\times\RR$, $\tilde\Phi(t,\cdot)$ maps a fiber of $\tilde\pi$ to a fiber of $\tilde\pi$.
In fact, in \cite{Cl77}, an action of the semigroup $[0,1]\times \RR$ on $\tilde N$ that preserves the fibration $\tilde\pi$ is constructed, and the present retraction is obtained by just restricting it to $[0,1]\times\{0\}\subset [0,1]\times\RR$. Similar to the proof of Proposition~\ref{p:Lamotke}, we can see that $\tilde\Phi$ descends to a strong deformation retraction $\Phi$ of $N$ onto $F_0$; the elementary details are omitted. 

If $N$ is smooth, the second property in the proposition is also a consequence of the action of the semigroup $[0,1]\times \RR$ on $\tilde N$ concretely given in the proof of \cite[Theorem~6.9]{Cl77} using local coordinates around the central fiber of $\pi$, especially from \cite[(6.8) in p.~244]{Cl77} where the multiplicities of the components of the central fiber are included.
If $N$ is singular, we consider the following diagram
\[
  \begin{CD}
H_2\big(\tilde N\big) @>>> H_2 \big(\tilde F_0\big) \\
@V{\mu_*}VV @VV{\mu_*}V\\
H_2(N)  @>>> H_2(F_0) 
  \end{CD}
\]
where the horizontal arrows are the isomorphisms induced by the deformation retractions from $\tilde N$ onto $\tilde F_0$ and from $N$ onto $F_0$ respectively.
This diagram is commutative as $\Phi = \mu\circ\tilde\Phi$.
The assertion (ii) for singular $Y$ readily follows from this diagram, using the validity for smooth $Y$.

The third property is elementary to see. Take the standard retraction 
\begin{align}
\phi:[0,1]\times (\ol\DDD_{\eee}\setminus\{0\})\rightarrow \ol\DDD_{\eee} \setminus \{0\}
\end{align}
of $\ol\DDD_{\eee} \setminus \{0\}$ onto $\ptl\DDD_{\eee}$, which is given by $\phi(t,z) = t z + \eee (1-t)z/|z|$, and let $V$ be the associated vector field on $\ol\DDD_{\eee} \setminus \{0\}$. The latter vanishes on the boundary circle $\ptl\DDD_{\eee}$.
Since $\pi:N\rightarrow\DDD$ is everywhere maximal rank over $\ol\DDD_{\eee} \setminus \{0\}$, the vector field $V$ can be lifted to a smooth vector field on $\pi\inv(\ol\DDD_{\eee} \setminus \{0\})$, using a partition of unity.
Then the flow of the lifted vector field gives a retraction of $\pi\inv(\ol\DDD_{\eee} \setminus \{0\} )$ onto $\pi\inv(\ptl \DDD_{\eee})$. 
\end{proof}
We will also need the following technical lemma. 
\begin{lemma}\label{l:b2rel} Let $\tilde{N}$ be a smooth complex surface and $\tilde\pi: \tilde{N} \rightarrow \DDD$ a proper elliptic fibration. Let $\tilde{F}=\tilde\pi\inv(0)$ and that $f$ is of maximal rank on $\tilde N \minus \tilde F $. 
Let $\bbb:H_2(\tilde N \minus \tilde F) \rightarrow H_2(\tilde N)$ be the natural homomorphism induced by the inclusion $\tilde N\minus\tilde F\subset \tilde N$.
Then the image of $\bbb$ is 1-dimensional and it is generated by the image of the fiber class of the elliptic fibration. 
\end{lemma}
\begin{proof}
We consider the long exact sequence of homology groups with real coefficients associated with the pair $(\tilde N,\tilde N\minus \tilde F)$, which is
\begin{align}\label{les:NF1}
\cdots\lras
H_3\big(\tilde N\big) \lras H_3\big(\tilde N,\tilde N\minus \tilde F\big)
\lras
H_2 \big(\tilde N\minus \tilde F\big) \stackrel{\bbb}\lras H_2\big(\tilde N\big) \lras \cdots.
\end{align}
We have 
\begin{align}
H_3(\tilde N) \simeq H_3(\tilde F)=0,
\end{align}
from Proposition \ref{p:Clemens}~(i) and $\dim_{\RR} \tilde F=2$. 
By Poincar\'e-Alexander-Lefschetz duality (see \cite[Theorem~VI.8.3]{Bredon}) and excision, we have 
\begin{align}
\label{e:PAD}
H_3(\tilde N, \tilde N\minus\tilde F)\simeq \check{H}^1(\tilde{F}) \simeq H^1(\tilde{F}),
\end{align}
where $\check{H}^1(\tilde{F})$ is the direct limit of $H^1(U)$ as $U$ ranges over all open subsets of $\tilde{F}$ in $\tilde{N}$, which is equal to the ordinary cohomology of $\tilde{F}$ 
from Proposition \ref{p:Clemens}~(i). We also have
\begin{align}
H_2(\tilde N \minus \tilde F) \simeq H_2(\ptl \tilde N),
\end{align}
from Proposition~\ref{p:Clemens}~(iii).
Therefore, from \eqref{les:NF1}, to show that $\bbb$ has a 1-dimensional image, it suffices to show
\begin{align}\label{b2rel1}
b_2(\ptl \tilde N) = b_1(\tilde F) + 1.
\end{align}
Let $\tilde N_0\rightarrow\DDD$ be the relatively minimal model of the elliptic fibration $\tilde N\rightarrow \DDD$, which is given as the composition of blowing downs of $(-1)$-curves in $\tilde F$.
This readily means that $b_1(\tilde F) = b_1(\tilde F_0)$, where $\tilde F_0$ is the fiber of $\tilde N_0\rightarrow\DDD$ over $0$. 
Since $ \ptl \tilde N\simeq\ptl \tilde N_0$
and also $b_2(\ptl \tilde N_0) = b_1(\ptl \tilde N_0)$ as $\ptl\tilde N_0$ is compact oriented and 3-dimensional, these mean that \eqref{b2rel1} is equivalent to 
\begin{align}\label{b2rel12}
b_1\big(\ptl\tilde N_0\big) = b_1\big(\tilde F_0\big) + 1.
\end{align}
Identifying $\ptl \tilde N_0$ with the mapping torus of the monodromy transformation $T^2\rightarrow T^2$ around the origin given by some matrix $A\in \SL(2,\ZZ)$, 
the Betti number $b_1(\ptl \tilde N_0)$ can be calculated from the long exact sequence associated with the mapping torus \cite[Chapter~2, Example~2.48]{Hatcher}, which is 
$$
H_1(T^2) \stackrel{A - E}\lras H_1(T^2) \lras H_1(\ptl \tilde N_0) \lras H_0(T^2)\lras 0.
$$
From this, using the concrete form of the matrix $A$ listed for example in \cite[p.~210, Table~6]{BHPV}, we obtain that $b_1(\ptl \tilde N_0)$ is explicitly given as shown in the following table:
\begin{center}
\begin{tabular}{c|c|c|c} 
   monodromy & type of the singular fiber & $b_1\big(\ptl \tilde N_0\big)$ & $b_1\big(\tilde F_0\big)$ \\ 
   \hline\hline
   $\{\id\}$ & $mI_0$ $(m>0)$ & 3 & 2\\
   \hline
   infinite 
   & $mI_b$ ($b\ge 1, m>0$) & 2 & 1 \\
   \hline
   infinite & $I_b^*$  ($b\ge 0$)& 1 & 0 \\ 
   \hline
   finite & $II, III, IV, II^*, III^*, IV^*$ & 1 & 0 
\end{tabular}
\end{center}
From this, we find that \eqref{b2rel12} holds for any type of the singular fiber.

To show the latter assertion, since the fiber class of the elliptic fibration $\tilde\pi:\tilde N\rightarrow \DDD$ belongs to the image of the fiber class of the fibration $\tilde N \minus \tilde F\rightarrow \DDD \minus\{0\}$ under $\bbb$, it suffices to see that the former class is non-zero.
Writing $\mf f$ for this class, from Proposition \ref{p:Clemens} (ii), 
by the natural homomorphism $H_2(\tilde N)\rightarrow H_2(\tilde F)$ induced by the deformation retraction of $\tilde N$ onto $\tilde F$, the class $\mf f$ is mapped to the class $\sum_{j} m_j [\tilde F_j]$, where $\mf f = \sum_j m_j \tilde F_j$ is the decomposition into the irreducible components.
Clearly this class is non-zero. Therefore, $\mf f$ is also non-zero. 
\end{proof}
\section{The case $a(Y)=1$}\label{s:a1case}
In this section, we show that in the setting of Theorem \ref{thm:main1}, $a(Y)$ cannot be one if the mapping $f:Z\rightarrow Y$ is surjective. Combined with Proposition~\ref{thm:main2} and Proposition~\ref{p:aneq0}, this will complete a proof of Theorem \ref{thm:main1}.
As in Section~\ref{s:Pre}, we always assume that the surface $Y$ is normal and the mapping $f$ has connected fibers. From Section~\ref{s:elliptic}, we have the following. 
\begin{proposition}\label{p:ell}
Let $f:Z\rightarrow Y$ be a surjective holomorphic mapping with connected fibers from a 3-dimensional connected compact complex manifold $Z$ satisfying $b_1(Z) = b_2(Z) = 0$ to a normal compact surface $Y$.
If $a(Y) = 1$, then $Y$ admits a surjective holomorphic mapping $\pi:Y\rightarrow \PP^1$ whose generic fiber is an elliptic curve. 
\end{proposition}
In the sequel, we always reserve $\pi:Y\rightarrow \PP^1$ for this elliptic fibration.
Further, noting that the discriminant locus of $f:Z\rightarrow Y$ as in \eqref{fb1} is necessarily contained in a finite number of fibers of $\pi$ from Proposition~\ref{p:a1s}, we make the following definition. 
\begin{definition}
\label{d:rsdef}
The subset $\Sigma\subset Y$ is the union of the following:
\begin{itemize}
\item
the entire fiber of $\pi$ that includes any component of the discriminant locus of $f:Z\rightarrow Y$,
\item
the entire fiber of $\pi$ that contains any singularity of $Y$.
\end{itemize}
The subset $D \subset Z$ is $D:=f\inv(\Sigma)$. Let $s$ denote the number of connected components of $\Sigma$ and let $r$ denote the number of irreducible components of $D$.
\end{definition}
Note that $\Sigma$ can include an irreducible component of a fiber of $\pi$ that is not in the discriminant locus of $f$.
If $\Sigma \neq \emptyset$, then since $\Sigma$ consists of entire fibers, we can write 
\begin{align}\label{Sigma2}
\Sigma = \pi\inv(\{u_1,\dots, u_s\})
\end{align}
for some distinct points $u_1,\dots, u_s\in\PP^1$. Note that 
\begin{align}
D = (\pi \circ f)^{-1} (\{ u_1,\dots, u_s\}),
\end{align}
so by Proposition $\ref{p:fd}$, all irreducible components of $D$ are $2$-dimensional. We let 
\begin{align}\label{Fi}
F_i:= \pi\inv(u_i), \quad 1\le i\le s.
\end{align}
Then $\Sigma = F_1\sqcup\dots\sqcup F_s$.
Each $F_i$ may or may not be a singular fiber of $\pi$.
Our next purpose is again to estimate the dimension of the relative cohomology group $H^2(Y,Y\minus\Sigma;\RR)$.
From now on, we will work on homology groups rather than cohomology groups. Since we only need cohomology groups with real coefficients, this does not affect our purpose.

Applying Proposition~\ref{p:Clemens} to the fibration $\pi:Y\rightarrow\PP^1$, for each $1\le i\le s$, we can choose a neighborhood $\DDD_i$ in $\PP^1$ of the point $u_i$ such that $\DDD_i\cap \DDD_j=\emptyset$ if $i\neq j$ and the inverse image 
$N_i:=\pi\inv(\DDD_i)$ admits a strong deformation retraction onto the fiber $F_i=\pi\inv(u_i)$ that is a lift of the standard retraction of $\DDD_i$ onto $u_i$.
Further, if $\mu:\tilde Y\lras Y$ is a very good resolution of all singularities of $Y$, then putting $\tilde\pi = \pi\circ\mu$,
$\tilde N_i:=\tilde\pi\inv(\DDD_i)$ admits a strong deformation retraction onto the fiber $\tilde F_i:=\pi\inv(u_i)$ that is a lift of the standard retraction of $\DDD_i$ onto $u_i$.
Moreover, we put $\tilde{\Sigma} := \tilde F_1 \cup \cdots \cup \tilde F_s$.
In the rest of this section, we use these notations.
Note that $\tilde F_i = F_i$ can hold and also $F_i$ can be a smooth fiber.

The following proposition is a key ingredient to our proof. 

\begin{proposition}\label{p:rs}
Let $\Sigma, D, r,$ and $s$ be as in Definition~\ref{d:rsdef}. 
If $\Sigma \neq \emptyset$, then
\begin{align}\label{sler}
s\le r,
\end{align}
with equality only if all the entire fibers $F_1,\dots, F_s$ in $\Sigma$ are irreducible (but not necessarily reduced). Furthermore, 
\begin{align}
\label{rs1}
b_1(Y\minus\Sigma) &= s-1, \\ 
\label{rs2}
b_1 (Z\minus D) &= r.
\end{align}
\end{proposition}
\begin{proof}
Since the inverse image under $f$ of each irreducible component of $\Sigma$ contains at least one irreducible component of $D$ of dimension $2$, 
the inequality \eqref{sler} and the necessary condition for equality follows immediately. 

For \eqref{rs1}, this is equivalent to $\dim H^1(Y\minus\Sigma) = s-1$.
From the isomorphism $\tilde Y\minus \tilde\Sigma \simeq Y\minus \Sigma$ induced by $\mu$, there is an isomorphism $H^1(Y\minus\Sigma)\simeq H^1(\tilde Y\minus\tilde\Sigma)$.
Lefschetz duality (see \cite[Theorem~6.2.19]{Spanier}) yields an isomorphism 
\begin{align}
H_1 (\tilde{Y} \minus \tilde{\Sigma}) \simeq \bar{H}^3(\tilde Y,\tilde\Sigma),
\end{align}
where $\bar{H}^3 (\tilde Y,\tilde\Sigma)$ is the direct limit of $H^3(\tilde{Y}, U)$ where $U$ ranges over all open neighborhoods $U$ of $\tilde\Sigma$ in $\tilde{Y}$. By Proposition \ref{p:Clemens}, and since we are using real coefficients, we obtain an isomorphism 
\begin{align}
H^1 (\tilde{Y} \minus \tilde{\Sigma}) \simeq H_3(\tilde Y,\tilde\Sigma).
\end{align}
The long exact sequence for homology groups associated with the pair $(\tilde Y,\tilde\Sigma)$, which is 
\begin{align}
\label{e:4les}
\cdots\lras H_3(\tilde Y) \lras H_3(\tilde Y,\tilde \Sigma)\lras H_2(\tilde \Sigma)
\stackrel{\aaa}\lras H_2(\tilde Y)\lras\cdots.
\end{align}
From Proposition~\ref{prop:b1v}, $b_1(\tilde Y) =0$. Since $\tilde Y$ is a compact smooth 4-manifold, Poincar\'e duality then implies that $H_3(\tilde Y)=0$. So from the exact sequence \eqref{e:4les}, we obtain
\begin{align}
H_3(\tilde Y,\tilde\Sigma)\simeq \Ker
\,[\aaa:H_2(\tilde \Sigma)\lras H_2(\tilde Y)
].
\end{align}
As $\tilde\Sigma$ is a disjoint union of $\tilde F_1,\dots,\tilde F_s$, we have a natural decomposition 
\begin{align}\label{rs3}
H_2(\tilde \Sigma)\simeq \oplus_{i=1}^s H_2(\tilde F_i).
\end{align}
Let $\mf f_i\in H_2(\tilde F_i)$ be the image of the fiber class of $\tilde\pi|_{\tilde N_i}:\tilde N_i\lras \DDD_i$ under the isomorphism $H_2(\tilde N_i)\simeq H_2(\tilde F_i)$ induced by the deformation retraction.
By Proposition \ref{p:Clemens}, if $\tilde F_i = \sum_{j}m_{ij} \tilde F_{ij}$ is the decomposition of the fiber into irreducible components, then $\mf f_i = 
\sum_{j}m_{ij} [\tilde F_{ij}]$, where $m_{ij}$ are the multiplicities. 
We will next show that, under the isomorphism \eqref{rs3}, 
\begin{align}\label{rs}
 \Ker
\,[\aaa:H_2(\tilde \Sigma)\to H_2(\tilde Y)
]=\big\{(c_1 \mf f_1, \dots, c_s \mf f_s)
\in \oplus_{i=1}^s H_2(\tilde F_i)\set c_1+ \dots + c_s = 0\big\},
\end{align}
which clearly means that $\dim (\Ker\aaa) = s-1$.
The inclusion ``$\supset$'' is immediate since $\aaa(\mf f_i)$ is the fiber class of $\tilde\pi$ for any $i$.
For the reverse inclusion, take any element $\xi\in \Ker\aaa$.
From~\eqref{rs3}, $\xi$ can be expressed as
$\xi = \sum_{i=1}^s \xi_i$, where $\xi_i\in H_2(\tilde F_i)$.
Then 
$\aaa(\xi) = \sum_{i=1}^s \aaa(\xi_i)$.
For the self-intersection number of $\aaa(\xi)$, since $\tilde F_i\cap \tilde F_j = \emptyset$ if $i\neq j$, 
\begin{align}\label{rs4}
\aaa(\xi)^2 = \sum_{i=1}^s \aaa(\xi_i)^2.
\end{align}
By Zariski's lemma, $\aaa(\xi_i)^2 \le 0$ for any $i$.
Hence, because $\aaa(\xi)^2 = 0$ as $\xi\in\Ker\aaa$,
from \eqref{rs4}, we obtain $\aaa(\xi_i)^2=0$ for any $i$.
Again by Zariski's lemma, this implies $\xi_i = q_i\mf f_i$ for some $q_i\in\QQ$.
Hence, $\xi = \sum_{i=1}^s q_i\mf f_i$, which immediately implies the inclusion ``$\subset$'' in \eqref{rs}. 
Therefore $b_1(Y\minus\Sigma) = s-1$.
 
Finally, we prove \eqref{rs2}. Since $b_1(Z) = b_2(Z) = 0$, the long exact sequence in cohomology for the pair $(Z, Z \minus D)$ implies that 
\begin{align}
H^1(Z \minus D) \simeq H^2(Z, Z \minus D). 
\end{align}
Since $Z$ is an oriented $6$-manifold, from \cite[Lemma~B.2]{Fulton}, we have that
\begin{align}
H^2(Z, Z \minus D) \simeq H_4^{BM}(D), 
\end{align}
where $H_4^{BM}(D)$ denotes the Borel-Moore homology of $D$; see \cite[Appendix B.2]{Fulton} for the definition.  By \cite[Lemma~B.4]{Fulton} (the proof of which holds in the analytic case), this is equal to the number of irreducible components of $D$ of dimension $2$,
which proves~\eqref{rs2}. 
\end{proof}

Continuing the estimation of $H^2(Y,Y\minus\Sigma)$, by excision, we have an isomorphism 
\begin{align}\label{rcs0}
H_2(Y,Y\minus\Sigma) \simeq \bigoplus_{i=1}^s H_2 (N_i, N_i\minus F_i),
\end{align}
where the coefficient rings are again $\RR$.
Let $E_i$ denote the exceptional curve of the birational mapping $\mu_i:=\mu|_{\tilde N_i}: \tilde N_i\lras N_i$ (with $E_i = \emptyset$ if $N_i$ is smooth).  
\begin{lemma} 
\label{l:b1}
If $\Sigma \neq \emptyset$, then for each $1 \leq i \leq s$, we have
\begin{align}
\label{e:leq0}
b_1(\tilde F_i) = b_1(F_i) + b_1(E_i) 
\end{align}
\end{lemma}
\begin{proof}
In the following, we omit the index $i$ for $N_i,F_i,\tilde F_i, \mu_i$
and also the coefficient ring $\RR$ for homology groups. If $\Sing Y \neq \emptyset$, let $p_1,\dots, p_k$ be all the singularities of $N$ and
$C_1,\dots, C_k$ the exceptional curves over $p_1,\dots, p_k$ respectively.
For simplicity, we put $C=C_1+\dots+C_k$ (so $E_i = C$ in the notation of the proposition). We denote $\ol F :=\tilde{F}/C$, which is the curve obtained from $\tilde{F}$ by identifying all points of $C$.  The quotient map $\tilde F \rightarrow \ol F$ factors as 
$$
\tilde F \stackrel{\mu}\lras F \lras \ol F,
$$
where the latter mapping is the quotient mapping that identifies $p_1,\dots, p_k$. 
The pair  $(F,\{p_1,\dots, p_k\})$ is a good pair, so from \cite[Proposition~2.22]{Hatcher}, we have the isomorphism
$H_q(F, \{p_1,\dots, p_k\} ) \simeq \tilde{H}_q(\ol F)$ for all $q \geq 0$, where $\tilde{H}_q$ denotes the reduced homology. Then 
the long exact sequence associated with the pair $(F,\{p_1,\dots, p_k\})$ is
\begin{align}
0 \lras H_1(F) \lras H_1 (\ol F) \lras  
\RR^k \lras H_0(F) \lras \tilde{H}_0 (\ol F) \lras 0.
\end{align}  
Since $F$ and $\ol F$ are connected, this yields the equality
\begin{align}
\label{e:leq1}
b_1( \ol{F}) = b_1(F) + k - 1. 
\end{align}
Next, the pair $(\tilde{F},C)$ is a good pair, so from \cite[Proposition~2.22]{Hatcher}, we have that 
$H_q(\tilde{F}, C) \simeq \tilde{H}_q(\ol F)$ for all $q \geq 0$. 
Since the second Betti number of any curve is equal to the number of irreducible components of the curve, the natural mapping from 
\begin{align}
H_2(\tilde{F}) \lras H_2( \tilde{F}, C) \simeq H_2( \ol F)
\end{align}
is clearly surjective. So the long exact sequence in homology for the pair $(\tilde{F},C)$ is 
\begin{align}
0 \lras H_1(C) \lras H_1(\tilde F) \lras H_1( \ol F) 
\lras H_0(C) \lras H_0(\tilde{F}) \lras \tilde{H}_0(\ol F) \lras 0.
\end{align}
Since $\tilde{F}$ and $\ol F$ are connected, This yields the equality 
\begin{align}
\label{e:leq2}
b_1( \tilde{F}) = b_1(C) + b_1( \ol F) -k +1.
\end{align}
Combining \eqref{e:leq1} and \eqref{e:leq2} yields \eqref{e:leq0}. 
\end{proof}
This lemma is used in the proof of the following. 
\begin{proposition}\label{p:b2relest}
If $\Sigma \neq \emptyset$, then for each $1\le i\le s$, 
$b_2(N_i,N_i\minus F_i)$ is at most
\begin{itemize}
\item
$b_2(F_i)$ if $N_i$ is smooth or has only rational tree singularities,
\item
$b_2(F_i) + 1$ if $N_i$ has a unicyclic singularity,
\item
$b_2(F_i) + 2$  if $N_i$ has an elliptic tree singularity.
\end{itemize}

\end{proposition}
\begin{proof}
Again, we omit the index $i$ and the coefficient ring $\RR$ for the homology groups.

First, if $N$ is smooth, then by a duality argument as in \eqref{e:PAD}, 
$H_2( N, N\minus F)\simeq H^2(F)$. This means $b_2(N,N\minus F) = b_2(F)$.
So the proof in this case is over. In the sequel, we assume that $N$ is singular and consider the long exact sequence of homology groups associated with the pair $(N,N\minus F)$, which is
\begin{multline}\label{les:NF4}
H_3(N)\lras H_3(N,N\minus F) \lras H_2(N\minus F) \lras H_2(N) \lras H_2(N,N\minus F)\\
\lras H_1(N\minus F) \lras H_1(N) \lras H_1(N,N\minus F)
\lras H_0(N\minus F) \lras H_0(N).
\end{multline}
In a similar way to the proof of Proposition \ref{p:rs}, using Proposition \ref{p:Clemens},
we obtain $H_q(N)\simeq H_q(F)$ for any $q\ge 0$,
$H_q(N\minus F)\simeq H_q(\ptl N)$ for any $q\ge 0$, and $b_2(N\minus F) = b_1(N\minus F)$. 
Further, since $N\minus F$ and $N$ are connected, the homomorphism
$H_0(N\minus F) \rightarrow H_0(N)$ is isomorphic and therefore
the homomorphism
$H_1(N) \rightarrow H_1(N,N\minus F)$ is surjective.
From these, the long exact sequence \eqref{les:NF4} yields
\begin{align}\label{NF2}
b_2(N,N\minus F) = b_2(F) - b_1(F) + b_3(N,N\minus F) + b_1(N,N\minus F).
\end{align}
As before, let $\mu:\tilde N\rightarrow N$ be a very good resolution of all singularities of $N$ and $\tilde F:=\mu\inv(F)$.
Then there is a commutative diagram
\[
  \begin{CD}
H_1(\tilde N\minus\tilde F) @>>> 
H_1(\tilde N) @>>> H_1(\tilde N,\tilde N\minus\tilde F) @>>>0 \\
 @V{\simeq}V{\mu_*}V    @VV{\mu_*}V @VV{\mu_*}V\\
   H_1(N\minus F) @>>> 
H_1(N) @>>> H_1(N,N\minus F) @>>>0
  \end{CD}
\]
where both rows are exact and the left vertical isomorphism is from the obvious isomorphism $\tilde N\minus \tilde F\simeq N\minus F$.
Since $\tilde N$ is smooth, we again have $H_1(\tilde N,\tilde N\minus\tilde F)=0$ from the duality argument.
Further, the middle vertical homomorphism $\mu_*$ is surjective since $\mu$ has connected fibers.
From this, we readily obtain that the homomorphism $H_1(N\minus F) \rightarrow
H_1(N)$ is surjective. From the second row, this means $H_1(N,N\minus F)=0$. Thus, we always have $H_1(N,N\minus F)=0$.

Next, we consider a similar commutative diagram, which is 
\[
  \begin{CD}
0 @>>>H_3 \big(\tilde N, \tilde N\minus\tilde F\big) @>>> 
H_2\big(\tilde N \minus \tilde F\big) @>{\bbb_i}>> H_2 \big(\tilde N\big) \\
@VV{\aaa_0}V @VV{\aaa_1}V    @V{\aaa_2}V{\simeq}V @VV{\aaa_3}V\\
0 @>>>   H_3(N, N\minus F) @>>> 
H_2(N\minus F) @>>> H_2(N) 
  \end{CD}
\]
By Lemma \ref{l:b2rel}, the image of the homomorphism $\bbb_i$ in the diagram is 1-dimensional which is generated by the image of the fiber class of the elliptic fibration.
Let $\tilde F=\sum _{1\le j\le n} m_j \tilde F_{j}$ be the decomposition of the fiber $\tilde F$ into the irreducible components.
By renumbering, we may suppose that for some $n'$ with $1<n'\le n$, the components $\tilde F_{n'},\dots, \tilde F_{n}$ are the ones that are contracted by $\mu:\tilde N\rightarrow N$.
Then since $\aaa_3$ is induced by $\mu$, we obtain 
\begin{align}
\aaa_3\Big( \sum _{1\le j\le n} m_j \big[\tilde F_{j}\big] \Big) = 
\sum _{1\le j\le n'-1} m_j \,\big[\mu (\tilde F_{j})\big].
\end{align}
Since $n'>1$, this means that the composition $\aaa_3\circ \bbb_i$ is not the zero mapping.  Replacing $H_2(\tilde N)$ with the image of $\bbb_i$, the first row in the diagram is still exact and $\aaa_3$ is then injective.
Therefore, by the Four Lemma (see \cite[p.~14]{MacLane}), the homomorphism $\aaa_1$ in the diagram is surjective.
Therefore,
\begin{align}\label{}
b_3(N,N\minus F) \le b_3(\tilde N,\tilde N\minus \tilde F)
= b_1(\tilde F),
\end{align}
where the equality is from duality.
Hence, using $H_1(N,N\minus F)=0$, we obtain from \eqref{NF2}
\begin{align}\label{NF3}
b_2(N,N\minus F) \le b_2(F) - b_1(F) + b_1(\tilde F).
\end{align}

We use this inequality to obtain the desired estimate.
If $N$ has only rational tree singularities, then by Proposition~\ref{p:scri},
every connected component of the exceptional divisor of $\mu$ satisfies $b_1=0$.
Therefore, by Lemma~\ref{l:b1}, $b_1(\tilde F) = b_1(F)$. If $N$ has an elliptic tree singularity, then again by Proposition~\ref{p:scri}, the exceptional divisor over it satisfies $b_1 = 2$. Further, $N$ has at most one elliptic tree singularity and cannot have a unicyclic singularity. By Lemma~\ref{l:b1}, 
because a rational tree singularity does not change $b_1$ as above,
this means  $b_1(\tilde F) = b_1(F) + 2$.
The same argument using Proposition~\ref{p:scri} and Lemma~\ref{l:b1} implies 
that if $N$ has a unicyclic singularity, then
$b_1(\tilde F) = b_1(F)+1$.
From these, the inequality~\eqref{NF3} means the desired estimate.
\end{proof}
From the decomposition \eqref{rcs0} and Proposition~\ref{p:b2relest},
we obtain
\begin{corollary}\label{cor:b2relY}
Let $n_e$ and $n_u$ be as in Definition~\ref{d:nenu}. 
Then we have the estimate
\begin{align}
b_2(Y,Y\minus\Sigma) \le b_2(\Sigma) +  2n_e + n_u.
\end{align}
\end{corollary}
Using all these results, we can now prove the main result of this section, which completes the proof of Theorem~\ref{thm:main1}.
\begin{proposition}
\label{p:aneq1}
Let $f:Z\rightarrow Y$ be a holomorphic mapping with connected fibers from a 3-dimensional connected compact complex manifold $Z$ satisfying $b_1(Z) = b_2(Z) = 0$ and $b_3(Z)\neq 2$ to a normal compact surface $Y$.
If $a(Y) = 1$, then $f$ is not surjective. 
\end{proposition}
\begin{proof}
Suppose that $f$ is surjective. From Proposition \ref{p:b2bdd0}, 
putting $f':=f|_{Z\minus D}$, this means 
\begin{align}\label{dir11}
b_2(Y)
\le b_2(Y,Y\minus\Sigma) - b_1(Y\minus\Sigma)
+ \dim H^0\big(Y\minus\Sigma,R^1f'_*\RR \big).
\end{align}
Further, from Corollary \ref{cor:b2relY}, we have
\begin{align}
b_2(Y,Y\minus\Sigma) \le b_2(\Sigma)  + 2 n_e + n_u.
\end{align}
If $\Sigma = \emptyset$, necessarily $\Sing Y = \emptyset$, and we have 
\begin{align}
b_2(Y) \leq 2.
\end{align}
But from \eqref{2ndbetti1} in Proposition \ref{p:b2}, since $Y$ is nonsingular, we have 
\begin{align}
b_2(Y)\ge 2p_g(Y) + 2.
\end{align}
Together, these imply that $p_g(Y)= 0$. Using that $Y$ admits a K\"ahler metric from Proposition~\ref{prop:b1v} (ii), by the Kodaira embedding theorem, this implies that the surface $Y$ is projective algebraic; see \cite{GH}. Hence $a(Y)=2$, which contradicts the assumption $a(Y) = 1$.

Therefore we can assume that $\Sigma \neq \emptyset$.   
From \eqref{rs1} in Proposition~\ref{p:rs}, we have $b_1(Y\minus\Sigma) = s-1$. Hence, from~\eqref{dir11}, we obtain 
\begin{align}\label{dir3}
b_2(Y) \le  b_2(\Sigma) + 2n_e + n_u -s +1 + \dim H^0(Y\minus\Sigma,R^1f'_*\RR).
\end{align}
Recall that $D \subset Z$ is defined by $D=f\inv(\Sigma)$.
From the 5 term exact sequence associated with the Leray spectral sequence
\begin{align}\label{Lss3}
E_2^{p,q} = H^p(Y\minus\Sigma,R^qf'_*\RR) \Longrightarrow H^{p+q}(Z\minus D,\RR),
\end{align}
we obtain the exact sequence 
\begin{align}\label{Lss4}
0 \lras H^1(Y\minus\Sigma,\RR) \lras H^1(Z\minus D,\RR) \stackrel{\aaa}\lras 
H^0(Y\minus\Sigma,R^1f'_*\RR). 
\end{align}
From \eqref{rs2} in Proposition~\ref{p:rs}, this means $s-1\le b_1(Z\minus D) = r$ and also
\begin{align}\label{sler1}
r-(s-1) \le \dim H^0(Y\minus\Sigma,R^1f'_*\RR).
\end{align}
Since $\dim H^0(Y\minus\Sigma,R^1f'_*\RR)\le 2$ as $\Sigma$ includes the discriminant locus of $f$, combining with $s\le r$ (see \eqref{sler}), we obtain 
\begin{align}\label{sler2}
s\le r \le s+1.
\end{align}

The homomorphism $\aaa$ in \eqref{Lss4} cannot be the zero mapping because 
it would imply $r= s-1$, which contradicts \eqref{sler2}.
So $H^0(Y\minus\Sigma,R^1f'_*\RR)\neq 0$. Suppose that the image of $\aaa$ is 1-dimensional. From \eqref{Lss4} and Proposition \ref{p:rs}, this means $s=r$, which implies that all the entire fibers $\Sigma_1,\dots,\Sigma_s$ are irreducible.
Hence, $b_2(\Sigma) = s$, so from \eqref{dir3}, we obtain
\begin{align}\label{dir4}
b_2(Y) \le  2 n_e + n_u + 1 + \dim H^0(Y\minus\Sigma,R^1f'_*\RR).
\end{align}
On the other hand, letting $\mu:\tilde Y\rightarrow Y$ be the minimal resolution of all singularities of $Y$, from \eqref{2ndbetti1} in Proposition \ref{p:b2},
we always have 
\begin{align}\label{tbp}
b_2(Y)\ge 2p_g(\tilde Y) + 2 n_e +  n_u + 2.
\end{align}
Using \eqref{dir4}, this implies 
$$
2p_g(\tilde Y) + 1 \le \dim H^0(Y\minus\Sigma,R^1f'_*\RR).
$$
Since $ \dim H^0(Y\minus\Sigma,R^1f'_*\RR)\le 2$, this means $p_g(\tilde Y) = 0$.
Using that $\tilde Y$ admits a K\"ahler metric from Proposition \ref{prop:b1v} (ii), again this implies that $a(Y) = 2$, which is a contradiction. 
Therefore, the image of $\aaa$ in \eqref{Lss4} is more than 1-dimensional.
It follows that $\dim H^0(Y\minus\Sigma,R^1f'_*\RR)=2$ and that $\aaa$ is surjective. So again from \eqref{Lss4}, $r= s+1$.
This means either 
\begin{itemize}
\item
all entire fibers $F_1,\dots, F_s$ in $\Sigma$ are irreducible and $f\inv(F_1),\dots, f\inv(F_s)$ have only one $2$-dimensional component except for one of these, which necessarily has exactly two $2$-dimensional components, or 
\item
all entire fibers $F_1,\dots, F_s$ are irreducible except for one, which consists of two irreducible components.
\end{itemize}

Suppose that the former is the case.
This includes the case where a fiber $f\inv(y)$ of $f$ has a 2-dimensional component for some $1\le i\le s$ and some point $y\in F_i$.
From the assumption, $b_2(\Sigma) = s$.
Hence, from \eqref{dir3}, as $\dim H^0(Y\minus\Sigma,R^1f'_*\RR)=2$ in the present situation,
$$
b_2(Y) \le 2n_e + n_u + 3.
$$
Again from \eqref{tbp}, this readily means $p_g(\tilde Y) = 0$, and this contradicts $a(Y)=1$ as above.

So suppose that the latter is the case.
In this case, since $b_2(\Sigma) = s+1$, \eqref{dir3} means
$$
b_2(Y) \le  2n_e + n_u +4.
$$
Also, as $\pi:Y\rightarrow \PP^1$ has a reducible fiber from the assumption, using the inequality \eqref{2ndbetti2},
\begin{align*}
b_2(Y)\ge 2p_g(\tilde Y)  + 2 n_e + n_u +  3.
\end{align*}
From these two inequalities, we obtain 
$2p_g(\tilde Y) \le 1$. This implies $p_g(\tilde Y) =0$,
which again contradicts $a(Y)=1$.
This means that the mapping $f:Z\rightarrow Y$ cannot be surjective.
\end{proof}

\bibliographystyle{amsalpha}
\bibliography{Fibrations}

\end{document}